\documentclass[12pt]{amsart}
\usepackage[all]{xy}
\usepackage{amssymb}
\usepackage{amsthm}
\usepackage{hyperref}
\hypersetup{colorlinks=true,linkcolor=blue,citecolor=magenta}
\usepackage{amsmath}
\usepackage{amscd,enumitem}
\usepackage{verbatim}
\usepackage{eurosym}
\usepackage{float}
\usepackage{color}
\usepackage{cases}
\usepackage{dcolumn}
\usepackage[mathscr]{eucal}
\usepackage[all]{xy}
\usepackage{hyperref}
\usepackage{bbm}
\usepackage[textheight=8.77in, textwidth=6.8in]{geometry}
\newtheorem*{thm*}{Theorem}
\newtheorem*{conj*}{Conjecture}

\newtheorem*{remark}{Remark}

\newtheorem{theorem}{Theorem}[section]

\newtheorem{lemma}[theorem]{Lemma}

\newtheorem*{ThreeRemarks}{Three Remarks}
\newtheorem*{FourRemarks}{Four Remarks}

\newtheorem{example}{Example}

\newcommand{\CL}{\mathrm{CL}}
\newcommand{\Z}{\mathbb{Z}}
\newcommand{\Q}{\mathbb{Q}}
\newcommand{\F}{\mathbb{F}}
\newcommand{\R}{\mathbb{R}}
\newcommand{\SL}{\operatorname{SL}}
\newcommand{\C}{\mathbb{C}}
\newcommand{\tor}{\mathrm{tor}}

\newcommand{\Vol}{\operatorname{Vol}}
\newcommand{\leg}[2]{\genfrac{(}{)}{}{}{#1}{#2}}
\newcommand{\cubic}[2]{\genfrac{[}{]}{}{}{#1}{#2}_3}
\newcommand{\cubiccharacter}[2]{\genfrac{(}{)}{}{}{#1}{#2}_3}
\numberwithin{equation}{section}

\begin{document}
\title[Elliptic curves and class numbers]{Quadratic twists of elliptic curves and class numbers}
\author{Michael Griffin, Ken Ono and Wei-Lun Tsai}
\address{Department of Mathematics, 275 TMCB, Brigham Young University, Provo, UT 84602}
\email{mjgriffin@math.byu.edu}
\address{Department of Mathematics, University of Virginia, Charlottesville, VA 22904}
\email{ken.ono691@virginia.edu}
\email{tsaiwlun@gmail.com}

\begin{abstract}  For positive rank $r$ elliptic curves $E(\Q)$, we employ
ideal class pairings 
 $$
E(\Q)\times E_{-D}(\Q) \rightarrow \CL(-D),
$$
for quadratic twists $E_{-D}(\Q)$ with
a  suitable ``small $y$-height'' rational point,
to obtain effective  class number lower bounds.
For the curves
$E^{(a)}: \ y^2=x^3-a,$ with rank $r(a),$ this gives 
$$
 h(-D) \geq \frac{1}{10}\cdot  \frac{|E_{\tor}(\Q)|}{\sqrt{R_{\Q}(E)}}\cdot  \frac{\pi^{\frac{r(a)}{2}}}{2^{r(a)}\Gamma\left (\frac{r(a)}{2}+1\right)} 
\cdot \frac{\log(D)^{\frac{r(a)}{2}}}{\log \log D},
$$
representing an improvement to the classical lower bound of Goldfeld, Gross and Zagier when $r(a)\geq 3$.
We prove that the number of twists $E_{-D}^{(a)}(\Q)$ with such a point 
(resp. with such a point and rank $\geq 2$ under the Parity Conjecture) is
$\gg_{a,\varepsilon} X^{\frac{1}{2}-\varepsilon}.$
 We give infinitely many cases where $r(a)\geq 6$.  These results can be viewed  as an analogue of 
the classical estimate of Gouv\^ea and Mazur for the number of  rank $\geq 2$ quadratic twists, where in addition we obtain
``log-power'' improvements to the Goldfeld-Gross-Zagier class number lower bound.
\end{abstract}

\thanks{The second author thanks the NSF (DMS-1601306) and
the Thomas Jefferson fund at the U. Virginia.}

\maketitle
\section{Introduction and statement of results}\label{Intro}

Originally posed by Gauss, the problem of obtaining effective lower bounds for class numbers $h(-D)$  of imaginary quadratic fields
 $\Q(\sqrt{-D}),$ which also count equivalence classes of
  integral  positive definite binary quadratic forms of fundamental discriminant $-D$, has been one of fundamental challenges in
number theory. In the 1930s,  Siegel \cite{Siegel} proved, for every $\varepsilon>0$, that there are constants $c_1(\varepsilon),
c_2(\varepsilon)>0$ for which
$$
c_1(\varepsilon) D^{\frac{1}{2}-\varepsilon} \leq h(-D) \leq c_2(\varepsilon)D^{\frac{1}{2}+\varepsilon}.
$$
Unfortunately, Siegel's lower bound is inexplicit;
there is no known formula for $c_1(\varepsilon)$. As a consequence, the problem of obtaining an effective nontrivial lower bound  remained open for many decades. Finally in the 1980s,
 Goldfeld, Gross and Zagier \cite{Goldfeld1, Goldfeld2, GrossZagier} solved this problem by making use of ideas and results related
 to the Birch and Swinnerton-Dyer Conjecture.  Thanks to the existence of an elliptic curve with analytic rank 3, Oesterl\'e
 \cite{Oesterle} used their work to establish the effective lower bound
 \begin{equation}\label{GGZ}
h(-D) >
\frac{1}{7000}\left(\log D\right) \prod_{\substack{p\mid D\ {\text {\rm prime}}\\ p\neq D}} \left(
1-\frac{[2\sqrt{p}]}{p+1}\right).
\end{equation}

In recent work \cite{GO1}, the first two authors
 obtained effective lower bounds that improve on (\ref{GGZ}) for certain polynomial families of discriminants. The
  method makes direct use of the arithmetic of elliptic curves.
The idea is to employ
{\it  ideal class pairings}, maps of the form
 $$
 E(\Q)\times E_{-D}(\Q)\rightarrow \CL(-D),
 $$
 where $E_{-D}$ is the $-D$-quadratic twist of  $E$. Such maps
 were first defined and studied by Buell, Call, and Soleng \cite{Buell, BuellCall, Soleng}.
 
 Suppose that $E/\Q$ is given by
$$
E: \ \ y^2 = x^3 +a_4x +a_6,
$$
where $a_4, a_6\in \Z$, with $j$-invariant $j(E)$ and discriminant $\Delta(E)$, and
suppose that $E(\Q)$ has Mordell rank $r=r_{\Q}(E)\geq 1$.
Throughout, we suppose that $-D<0$ denotes a negative fundamental discriminant.
We let  $E_{-D}/\Q$ be its $-D$-quadratic twist\footnote{For reasons which will become apparent later, we use this nonstandard normalization.} given by the model
\begin{equation}\label{twist}
E_{-D}: \  -D\cdot \left (\frac{y}{2}\right )^2=\, x^3 +a_4x+a_6.
\end{equation}
Suppose that $Q_{-D}=\left(\frac{u}{w^2},\frac{v}{w^3}\right)\in E_{-D}(\Q),$ where\footnote{We note that these hypotheses guarantee that
$v$ is even when $-D$ is odd.} 
 $uv\neq 0.$ In Section~\ref{PairingSection}, we recall Theorem~\ref{ThmQF}, which gives the explicit construction of the pairing. Moreover, the theorem
determines situations where the classes obtained by
 pairing  points in $E(\Q)$ with $Q_{-D}$ are inequivalent, thereby providing  a lower bound for $h(-D)$.

We use this idea to derive lower bounds for $h(-D)$ in terms of $\Omega_r:=\pi^{\frac{r}{2}}/\Gamma\left (\frac{r}{2}+1\right)$, the volume of the $\R^r$-unit ball, the regulator $R_{\Q}(E)$,  the diameter $d(E)$ (see (\ref{diameter})), the torsion subgroup $E_{\tor}(\Q)$ and the point $Q_{-D}.$ We define the natural constants
\begin{equation}\label{cE}
c(E):=\frac{|E_{\tor}(\Q)|}{2^{r+1}\sqrt{R_{\Q}(E)}}\cdot \Omega_r,
\end{equation}
and
\begin{equation}\label{cED}
c(E,Q_{-D}):=c(E) \cdot \prod_{\substack{p \text{ prime}\\p\mid w}}\left(1-\frac{1}{|E(\F_p)|}\right).
\end{equation}
Here $|E(\F_p)|$ denotes the number of $\F_p$-points on the reduction of $E$ modulo $p$
(even for primes of bad reduction), including the point at infinity.

Our first result is a generalization of Theorem 1.1 of \cite{GO1}, which uses the usual logarithmic heights  (see Section~\ref{HeightStuff}) of $j(E)$ and $\Delta(E),$  to define
\begin{equation}
\delta(E):={\color{black} \frac{1}{2}h_{W}(j(E))+\frac{1}{3}h_{W}(\Delta(E))+\frac{20}{3}}.
\end{equation}
To facilitate the comparison with $\log(D)$, we define
\begin{equation}\label{Tconstant}
T_E(-D,Q_{-D}):= \log \left(\frac{D}{|u|+w^2}\right)-\delta(E).
\end{equation}
Finally, we say that $Q_{-D}=\left(\frac{u}{w^2},\frac{v}{w^3}\right)\in E_{-D}(\Q)$ is {\it suitable for $E$} if $uv\neq 0$ and
\begin{equation}\label{suitable}
(|u|+w^2)\exp(\delta(E)+d(E)) < D < \frac{(|u|+w^2)^2\max(|u|,w^2)^2}{v^4}.
\end{equation}
It turns out that suitable rational points $Q_{-D}$  will have ``small'' $v.$  For notational convenience,
we let $\widehat{c}(E,Q_{-D}):= 2\cdot 3^r r \sqrt{d(E)}c(E)\cdot \mathcal{S}(w),$ where $\mathcal{S}(w)$ denotes the number
of positive square-free divisors of $w$.
We obtain the following theorem.
\begin{theorem}\label{General}
Assuming the notation and hypotheses above, if $Q_{-D}$ is suitable for $E$, then
$$
h(-D)\geq c(E,Q_{-D}) \cdot T_E(-D,Q_{-D})^{\frac{r}{2}}-
 \widehat{c}(E,Q_{-D})  \cdot  T_E(-D,Q_{-D})^{\frac{r-1}{2}}.
$$
\end{theorem}

To give an indication of the frequency that Theorem~\ref{General} offers an improvement to (\ref{GGZ}), we
consider the elliptic curves
\begin{equation}\label{jzero}
E^{(a)}: \ \  y^2 = x^3 -a,
\end{equation}
where $a$ is a positive integer. 
By constructing explicit infinite order points $Q_{-D}\in
E_{-D}^{(a)}(\Q),$ which are often suitable,
we obtain  effective lower bounds for $h(-D),$ formulated in terms of the rank
of $E^{(a)}$ and the natural constant $c(E^{(a)})$ defined in (\ref{cE}). For notational convenience, we let
\begin{equation}
\mathfrak{S}(E):= \{ -D  \ : \ E_{-D}(\Q)\ {\text {\rm has an infinite order suitable point}}\ Q_{-D}\}.
\end{equation}

\begin{theorem}\label{MainTheorem} If $E^{(a)}(\Q)$ has rank $r$ and $\varepsilon>0$, then we have
$$
\#\left  \{ -X<-D<0 \ : \  -D\in \mathfrak{S}(E^{(a)})\ {\text {\rm and}}\ 
 h(-D) > \frac{c(E^{(a)})}{5}\cdot   \frac{\log(D)^{\frac{r}{2}}}{\log \log D}\right\} \gg_{a,\varepsilon} X^{\frac{1}{2}-\varepsilon}.
$$
 Assuming the Parity Conjecture for elliptic curves, 
 we may also require $r_{\Q}(E_{-D}^{(a)})\geq 2.$
\end{theorem}

\begin{FourRemarks}  \ \

\noindent
(1) The multiplicative constant $1/5$ was chosen for aesthetics, and is the lower bound offered in the abstract. By Lemma~\ref{lambda}, one can replace $1/5$ with
 any constant $<0.2158.$
\smallskip

\noindent (2)
The lower bound  $\gg_{a,\varepsilon} X^{\frac{1}{2}-\varepsilon}$ for the number of discriminants $-X<-D<0$ in Theorem~\ref{MainTheorem} is an improvement of the results in \cite{GO1} (see the second Remark after Theorem 1.2 in \cite{GO1}), where
$\gg X^{\frac{1}{3}}$ many discriminants are obtained.

\smallskip
\noindent 
(3)  Each rank 1 curve $E^{(a)}(\Q)$ gives $\gg_{a,\varepsilon} X^{\frac{1}{2}-\varepsilon}$ many discriminants $-X<-D<0$
with
$$
 h(-D) > \frac{c(E^{(a)})}{5}\cdot   \frac{\sqrt{\log D}}{\log \log D}.
$$
Although such estimates do not improve on (\ref{GGZ}), it is plausible that a new proof of Gauss's class number 1 problem,  famously proved by Baker, Heegner and Stark \cite{Baker, Heegner, Stark}, can be obtained 
by making use of the large supply of rank 1 curves $E^{(a)}(\Q)$.

\smallskip
\noindent
(4) Goldfeld's famous conjecture \cite{GoldfeldConjecture} on quadratic twists of elliptic curves implies that asymptotically ``half of the quadratic twists'' of $E^{(a)}$
have rank 0 (resp. 1). 
Theorem~\ref{MainTheorem} is related to the well-studied problem of
 estimating the number of those rare twists with rank $\geq 2$. In an important paper,
Stewart and Top (see Theorem~3 of \cite{StewartTop}) unconditionally proved 
$$\# \{ -X<-D<0 \ : \ r_{\Q}(E^{(a)}_{-D})\geq 2\} \gg_{a} \frac{X^{\frac{1}{7}}}{\log X}.
$$
However, it is widely believed that this lower bound is not optimal. Indeed, 
a classical result of
Gouv\^ea and Mazur (see Theorem~2 of \cite{GouveaMazur})), which assumes the Parity Conjecture, gives
$$\# \{ -X<-D<0 \ : \ r_{\Q}(E^{(a)}_{-D})\geq 2\} \gg_{a,\varepsilon} X^{\frac{1}{2}-\varepsilon}.
$$
If $r(a)\geq 3$, then Theorem~\ref{MainTheorem} can be viewed as an analogue of this result, where we also obtain a
``log-power'' improvement to the Goldfeld-Gross-Zagier class number lower bound (\ref{GGZ}).
\end{FourRemarks}

\begin{example}
The elliptic curve\footnote{All computations in this paper were performed using \texttt{SageMath} \cite{sage}.} $E^{(174)}(\Q)$ has no nontrivial torsion, and has rank $3$, with generators
$(7,13), (25/4, 67/8),$ and  $(151/25 , -851/125).$  Moreover, we have 
$\Omega_3=4\pi/3\approx 4.1887$ and $R_{\Q}(E^{(174)})\approx 46.1056,$ which gives
$c(E^{(174)}) \approx 0.0385>1/26.$
Therefore, Theorem~\ref{MainTheorem} implies that
$$
\#\left  \{ -X<-D<0 \ : \  -D\in \mathfrak{S}(E^{(174)})\ {\text {\rm and}}\ 
 h(-D) > \frac{1}{130}\cdot  \frac{\log(D)^{\frac{3}{2}}}{\log \log D}\right\}\gg_{\varepsilon} 
 X^{\frac{1}{2}-\varepsilon}.
$$
Assuming the Parity Conjecture, we may also require that $r_{\Q}(E^{(174)}_{-D})\geq 2$.
\end{example}

\begin{example}Theorem~\ref{MainTheorem} holds for infinite (if any) subgroups of $E^{(a)}(\Q)$
(see Lemma~\ref{PropBoundsApprx}),
where one employs the natural analogues of $R_{\Q}(E^{(a)}).$ 
Elkies \cite{Elkies1, Elkies2} found that $E^{(k)}$, where
$$
  k := 2195745961 \cdot 413891567044514092637683,
$$
has $r_{\Q}(E^{(k)})\geq 17.$ Theorem~\ref{MainTheorem} for this curve
then gives
$$
\#\left  \{ -X<-D<0 \ : \ -D\in \mathfrak{S}(E^{(k)}) \ {\text {\rm and}}\ 
 h(-D) > \frac{c(E^{(k)})}{5}\cdot \frac{\log(D)^{\frac{17}{2}}}{\log \log D}\right\} \gg_{k,\varepsilon} X^{\frac{1}{2}-\varepsilon},
$$
offering a large ``log-power'' improvement to (\ref{GGZ}).
Using the 17 independent points listed in
\cite{Elkies1, Elkies2}, one can find that $c(E^{(k)})\approx 2.84243\cdot 10^{-19}.$  Again, assuming the Parity Conjecture, we may also require that $r_{\Q}(E^{(k)}_{-D})\geq 2$.
\end{example}

\begin{example}
For $r\in \{3, 4, 5, 6\},$ we consider
curves $E^{(a_r(T))}/\Q(T),$
where 
\begin{align*}
a_3(T)&:= 2^43^3(4T^6-8T^4+40T^2-31),\\  \ \ \\
a_4(T)&:=6075T^{12}+38070T^{11}+81513T^{10}+83106T^9+67797T^8+39528T^7+27270T^6\\
&\ \ \ \ \ \ \ \ \ +58968T^5
+89181T^4+84834T^3+52353T^2+23814T-9261,\\ \ \ \\
a_5(T)&:=\frac{64}{27}\left(T^{18}+2973T^{12}-369249T^6+11764900\right),\\ \ \ \\
a_6(T)&:= ({2^6\cdot7^{54}\cdot13^2\cdot1297 \cdot 74449^3 \cdot 793041539 \cdot 1995792099060563}/{27})\cdot T^{54}\\ 
&\ \ \ \ \ \ \ \ \ +(2^9 \cdot7^{53} \cdot 13 \cdot 1999 \cdot 74449^2 \cdot 1923403 \cdot 881277323405000103687971
)\cdot T^{53}+\cdots\notag\\
&\ \ \ \ \ \ \ \ \ +\cdots +\cdots+(2^9 \cdot7^{53} \cdot 13 \cdot 1999 \cdot 74449^2 \cdot 1923403 \cdot 881277323405000103687971
)\cdot T\notag\\
&\ \ \ \ \  \ \ \  \ +(2^6 \cdot 5^3 \cdot 11 \cdot 8123 \cdot 1882419814724639 \notag\\
&\ \ \ \ \ \ \ \ \  \  \ \ \   \cdot 177610817485358112101332029225675499667600288403153465585113540179/27).\notag
\end{align*}
Using work of Mestre \cite{Mestre}, Stewart and Top \cite{StewartTop}  proved that
each $E^{(a_r(T))}/\Q(T)$ has rank $r.$
By Silverman's specialization theorem \cite{Silverman}, 
for all but finitely many integers $t,$ Theorem~\ref{MainTheorem} gives
$$
\#\left\{ -X<-D<0 \ : \  -D\in \mathfrak{S}(E^{(a_r(t))})\ {\text {\rm and}}\ h(-D)>\frac{c_r(t)}{5}\cdot
\frac{\log(D)^{\frac{r}{2}}}{\log \log D}\right\} \gg_{r,\varepsilon} X^{\frac{1}{2}-\varepsilon},
$$
where  $c_r(t)$ is defined using the $r$ points
given in \cite{StewartTop}.
Again, assuming the Parity Conjecture, we may also require that $r_{\Q}(E^{(a_r(t))}_{-D})\geq 2$.
\end{example}

This paper is organized as follows. In Section~\ref{NutsAndBolts} we prove Theorem~\ref{General}, an extension of Theorem 1.1 of \cite{GO1}. To prove Theorem~\ref{MainTheorem}, we use the fact
that the existence of a suitable point $Q_{-D}\in E_{-D}^{(a)}(\Q)$ for $E^{(a)}$ is equivalent to the 
solvability of the
Diophantine equation
$$-Dt^2=m^3-an^6,
$$
where the integer triples $(m,n,t)$ satisfy certain inequalities. To make use of this fact, in Section~\ref{DiophantineResult}
we prove an auxiliary theorem of independent interest (see Theorem~\ref{summatory}), which gives asymptotic formulas 
for the number of solutions to this equation where the parameters are chosen from natural intervals.
In particular, as a function of $T=T(X)$, this theorem can be used to estimate (see (\ref{refined}))
the number of discriminants $-X<-D<0$ for which there is an infinite order rational point $Q_{-D}=\left(\frac{u}{w^2},\frac{v}{w^3}\right)
\in E_{-D}(\Q)$ with $T\leq v \leq 2T.$
This result is obtained using
 the P\'olya-Vinogradov inequality, combined with a sieve-type count involving solutions to polynomial congruences.   In Section~\ref{MainTheoremProof}, we then prove Theorem~\ref{MainTheorem}.

\section*{Acknowledgements}  The second author thanks the NSF (DMS-1601306) and
the Thomas Jefferson fund at the U. Virginia. The authors thank the referee, N. Elkies, D. Goldfeld, B. Gross, J. Iskander, F. Luca, K. Soundararajan, D. Sutherland and J. Thorner for  useful comments concerning this paper.

\section{Ideal class pairings and the proof of Theorem~\ref{General}}\label{NutsAndBolts}

Works by Buell, Call, and Soleng \cite{Buell, BuellCall, Soleng} offered elliptic curve ideal class pairings, which
produce discriminant $-D$ integral positive definite binary quadratic forms from points on
$E(\Q)$ and $E_{-D}(\Q)$.
Theorem 2.1 of \cite{GO1} is a generalization and minor correction of Theorem 4.1 of \cite{Soleng}.\footnote{This corrects sign errors in the discriminants in Theorem 4.1 of \cite{Soleng}, and also ensures the resulting quadratic forms are integral when $C\neq 1$. Moreover, this theorem allows for both even and odd discriminants.}  
We begin by recalling this result.

\subsection{Ideal class pairing}\label{PairingSection}
Assume the notation from Section~\ref{Intro}.
Let $P=(\tfrac{A}{C^2},\tfrac{B}{C^3})\in E(\Q),$ with $A, B, C\in \Z$, and $Q=(\tfrac{u}{w^2},\tfrac{v}{w^3})\in E_{-D}(\Q),$
with $u,v,w\in \Z,$ not necessarily in lowest terms, but so that no sixth power divides $\gcd(u^3,v^2,w^6)$. Every $Q$ clearly has such a representation, and thanks to (\ref{twist}), we find that $\gcd(u,w^2)$ and $\gcd(v,w^3)$ both divide $D$. Moreover, suppose that
 $uv\neq 0$, which guarantees that $v$ is even when $-D$ is odd.
If we let $\alpha:=|Aw^2-u C^2|$ and  $G:=\gcd(\alpha, C^6v^2),$ then there are integers $\ell$
for which $F_{P,Q}(X,Y)$ defined below is a discriminant $-D$ positive definite integral binary quadratic form. 
\begin{equation}
{\color{black}
F_{P,Q}(X,Y):=\frac{\alpha}{G} \cdot X^2+\frac{2w^3 B+\ell  \cdot \tfrac{\alpha}{G}}{C^3v}\cdot XY +\frac{\left({ 2w^3B+
\ell \cdot \tfrac{\alpha}{G}}\right)^2+C^6 v^2{D}}{4C^6v^2\cdot \frac{\alpha}{G}}\cdot Y^2.
}
\end{equation}

\begin{theorem}\label{ThmQF}{\text {\rm [Theorem 2.1 of \cite{GO1}]}}
Assuming the notation and hypotheses above, $F_{P,Q}(X,Y)$ is well defined (e.g. there is such an $\ell$)  {\color{black} in $\CL(-D)$}.
Moreover, if $(P_1,Q_1)$ and $(P_2, Q_2)$ are two such pairs for which $F_{P_1,Q_1}(X,Y)$ and $F_{P_2,Q_2}(X,Y)$ are $\SL_2(\Z)$-equivalent, then $\frac{\alpha_1}{G_1}=\frac{\alpha_2}{G_2}$ or $\frac{\alpha_1\alpha_2}{G_1G_2}\geq D/4$.
\end{theorem}

\begin{example} For $E: y^2=x^3-4x+9$, we have points $P_1:=(0,3)$ and $P_2:=(-2,3)$. We consider
the example of $h(-24)=2$.
Using $Q:=(-3,1)\in E_{-24}(\Q)$ and $\ell=2$, we obtain representatives for the two inequivalent discriminant $-24$ forms
$$F_{P_1,Q}(X,Y)=3X^2+12XY+14Y^2 \ \ \ {\text {\rm and}}\ \ \ F_{P_2,Q}(X,Y)=X^2+8XY+22Y^2.
$$
\end{example}

\subsection{Proof of Theorem~\ref{General}}\label{HeightStuff}

To deduce Theorem~\ref{General} from Theorem~\ref{ThmQF}, we use estimates for the number of bounded height rational points on elliptic
curves. We recall the facts we require.
Each rational point $P\in E(\Q)$ has the form $P=(\frac{A}{C^2},\frac{B}{C^3})$, with $A,B,C$ integers such that $\gcd(A,C)=\gcd(B,C)=1$. The {\it naive height} of $P$ is
$H(P)=H(x):=\max(|A|, |C^2|),$ and
 the {\it logarithmic height} (or Weil height) is 
$h_W(P)= h_W(x):=\log H(P).$
Finally, we recall the {\it canonical height }
\begin{equation}
\widehat h(P):= \tfrac{1}{2}\lim_{n\to \infty}\frac{h_W(nP)}{n^2}.
\end{equation}

The following theorem of  Silverman~\cite{Silverman} bounds the difference between the logarithmic and canonical heights
of rational points in terms
of the logarithmic heights of $j(E)$ and $\Delta(E)$. 

\begin{theorem}[Theorem 1.1 of \cite{Silverman}]\label{Silverman_bounds}
If $P\in E(\Q)$, then
\[
-\tfrac{1}{8}h_W(j(E))-\tfrac{1}{12}h_W(\Delta(E))-0.973\leq \widehat{h}(P)-\frac{1}{2}h_W(P)\leq \tfrac{1}{12}h_W(j(E))+\tfrac{1}{12}h_W(\Delta(E))+1.07.
\]
\end{theorem}

Asymptotics for the number of rational points on an elliptic curve with bounded height are well known
(for example, see \cite[Prop 4.18]{Knapp}). If
 $E(\Q)$ has rank $r\geq 1$ and
$\Omega_r:=\pi^{\frac{r}{2}}/\Gamma\left (\frac{r}{2}+1\right)$,  then in terms of the
 regulator $R_{\Q}(E)$ and $|E_{\tor}(\Q)|$, we have
\begin{equation}
\# \{P\in E(\Q) \mid \widehat h(P) \leq T/4\}
\sim \frac{|E_{\tor}(\Q)|}{2^r\sqrt{R_{\Q}(E)}} \cdot \Omega_r T^{\frac{r}{2}}.
\end{equation}

To prove Theorem~\ref{General}, 
we require effective lower bounds for the number of points with bounded height, which is essentially
the problem of counting lattice points in $r$-dimensional spheres.
To this end, we let $B(R)$ denote the closed ball in $\R^r$ of radius $R$ centered at the origin.
Furthermore, if $\mathcal P$ is any parallelepiped,  then let $d(\mathcal P)$ denote its (squared) diameter,  the largest square-distance between any two vertices. 
In our setting, if $\{P_1,\dots,P_r\}$ is a basis of $E(\Q)/E_{\tor}(\Q)$, then the (squared) diameter is 
\begin{equation}\label{diameter}
d(E):=\max_{\delta_i\in \{\pm1,0\}} 2\widehat h\left(\sum_{i=1}^r\delta_iP_i\right ).
\end{equation}
This is the diameter of the parallelepiped in $\R^r$ constructed from vectors $\textbf{v}_1,\dots \textbf{v}_r,$ where $\textbf{v}_i\cdot \textbf{v}_j=\langle P_i,P_j\rangle:=
{\color{black}\frac{1}{2}\left(\widehat{h}(P_i+P_j)-\widehat{h}(P_i)-\widehat{h}(P_j)\right).}$

\begin{lemma}\label{PropRawBounds} Let $\Lambda$ be a lattice in $\R^r$ of full rank, and let $\mathcal P$ be any fundamental parallelepiped of $\Lambda$. If $T>4d(\mathcal P),$ then we have 
\[
\left| \frac{2^r\Vol{\mathcal P}}{\Omega_r}\cdot \# \{\Lambda \cap B(\tfrac12 T^{\frac12})\} - T^{\frac{r}{2}}\right|\ \leq  \ \ 3^rT^{\frac{r-1}{2}}d(\mathcal P)^{\frac{1}{2}}.
\]
\end{lemma} 

\begin{proof}
Let $\left\{\mathbf v_1,\mathbf v_2,\dots, \mathbf v_r\right\}$ be a basis for $\Lambda,$
and let $\mathbf{w}:=\sum_{i=1}^r \tfrac12 \mathbf{v}_i.$  For each point $\lambda \in \Lambda$, let $P_\lambda$ be the half-open parallelepiped given by 
\[\mathcal P_\lambda=\left \{\lambda +\sum_{i=1}^r x_i \mathbf{v}_i ~\mid~ x_i\in [0,1)\right\}.\]
If $\mathcal P_\lambda$ intersects the shifted ball $B(\tfrac12 T^{\frac{1}{2}}-\tfrac12d^{\frac{1}{2}})+\mathbf{w},$ then $\lambda\in B(\tfrac12 T^{\frac{1}{2}}).$  Therefore, we have
\begin{eqnarray*}
\# \left(\Lambda \cap B(\tfrac12 T^{\frac{1}{2}}) \right)&\geq \frac{\Vol \left(B(\tfrac12 T^{\frac{1}{2}}-\tfrac12d^{\frac{1}{2}})\right)}{\Vol (\mathcal P_\lambda)}
= \frac{\Omega_r}{2^r\Vol (\mathcal P_\lambda)}\cdot \left(T^{\frac{1}{2}}- d^{\frac{1}{2}} \right)^r.
\end{eqnarray*}
On the other hand, If $\lambda\in B(\tfrac12 T^{\frac{1}{2}})$, then $\mathcal P_\lambda$ is contained in the shifted ball $B(\tfrac12 T^{\frac{1}{2}}+\tfrac12d^{\frac{1}{2}})+\mathbf{w}.$ Therefore, we have
\begin{eqnarray*}
\# \left(\Lambda \cap B(\tfrac12 T^{\frac{1}{2}}) \right)&\leq \frac{\Vol \left(B(\tfrac12 T^{\frac{1}{2}}+\tfrac12d^{\frac{1}{2}})\right)}{\Vol (\mathcal P_\lambda)}
= \frac{\Omega_r}{2^r\Vol (\mathcal P_\lambda)}\cdot \left(T^{\frac{1}{2}}+ d^{\frac{1}{2}} \right)^r.
\end{eqnarray*}
We now apply the approximation 
\[\left(x+y\right)^r\leq x^r+ b^{-1} x^{r-1}y\left(\left(1+b\right)^r-1\right)<x^r+ b^{-1}x^{r-1}y\left(1+b\right)^r,\] whenever $x$ and $y$ are positive and $0<y/x<b<1.$ By hypothesis, we have $\sqrt{T}>2\sqrt{d(\mathcal P)},$  and
so the conclusion follows by letting  $b=1/2.$
\end{proof}

We use this lemma to count bounded height rational points on an elliptic curve, whose coordinates have denominators that  satisfy certain coprimality conditions.
Namely, suppose that
 $Q_{-D}=\left(\frac{u}{w^2}, \frac{v}{w^3}\right)\in E_{-D}(\Q)$ is as in Theorem~\ref{General}.
Recalling the constant $c(E,Q_{-D})$ defined in (\ref{cED}), we have the following
lemma which counts the points with bounded height on $E(\Q)$ with denominators that are coprime to $w$.

 \begin{lemma}\label{PropRawBoundsCop}
 Assume the notation and hypotheses in Theorem~\ref{General}. If  
$T>4d(E),$ then 
\begin{align*}
\# \{P=\left(\tfrac{A}{C^2},\tfrac{B}{C^3}\right)\in E(\Q) \mid \widehat h(P) \leq T/4 \, \ &{\text{{\rm and}}} \, \gcd(C,w)=1 \}\\
&\geq c(E,Q_{-D})\cdot T^{\frac{r}{2}}-\widehat c(E,Q_{-D})\cdot T^{\frac{r-1}{2}}.
\end{align*}
 \end{lemma}

\begin{proof}
Let $\mathcal B:=\{P_1,\dots,P_r\}$ be any basis for $E(\Q)$, and consider linearly independent vectors $v_1,v_2,\dots, v_r\in \R^r$ for which $v_i\cdot v_j=\langle P_i,P_j\rangle.$ Let $\psi:E(\Q) \to \R^r$ be the additive homomorphism defined so that $\psi(P_i)=v_i$, and $\psi(E_{\tor}(\Q))$ is the origin. 

For any integer $n$, let 
\begin{equation}
E_{\langle n\rangle}:=\{P=\left(\tfrac{A}{C^2},\tfrac{B}{C^3}\right)\in E(\Q) \mid %\widehat h(P) \leq T \ \  \text{and} \ \ 
C\equiv0~(\mathrm{mod~} n) \}.
\end{equation}
If $p$ is prime, then $E_{\langle p\rangle}$ is the kernel of the reduction modulo  $p$  map $E(\Q)\to E(\F_p),$ which includes the point at infinity. This set is closed under addition, and the image 
$\Lambda_p:=\psi(E_{\langle p\rangle})$ is a lattice. If $p$ is a prime of good reduction for $E$, then this follows since the reduction map is a group homomorphism. Otherwise, the claim is straightforward to confirm directly with the definition of the group law. 
More generally, $\Lambda_n:=\psi(E_{\langle n\rangle})$ is a lattice for any square-free integer $n$, since 
$\Lambda_n=\bigcap_{p\mid n}\Lambda_p.$

By the Nagell-Lutz Theorem, $E_{\langle n\rangle} \cap E_{\tor}(\Q)$ is trivial if $n>1$, and so $\psi$ is injective on $E_{\langle n\rangle}$. Thus, by an inclusion/exclusion argument we have that 
\begin{multline}\label{EqIncExc1}
\# \{P=\left(\tfrac{A}{C^2},\tfrac{B}{C^3}\right)\in E(\Q) \mid \widehat h(P) \leq T/4 \, \text{and} \, \gcd(C,w)=1 \}\\
=|E_\tor(\Q)| \cdot \# (\Lambda_1\cap B(\tfrac12 T^{\frac12})) \,-\, \sum_{p\mid w} \# (\Lambda_p\cap B(\tfrac12 T^{\frac12})) \, +\,
\sum_{\substack{p,q\mid w\\p\neq q}} \# (\Lambda_{pq}\cap B(\tfrac12 T^{\frac12})) \,-\, \dots,
\end{multline}
where the sums are over prime divisors of $w$. 

By Lemma~\ref{PropRawBounds}, 
we have that 
\[
\left| \frac{2^r\Vol({\mathcal P_n})}{\Omega_r}\cdot \# \{\Lambda \cap B(\tfrac12 T^{\frac12})\} - T^{\frac{r}{2}}\right|\ \leq  \ \ 3^rT^{\frac{r-1}{2}}d_n^{\frac{1}{2}} ,
\]
where $d_n$ is the minimum diameter of any choice of $\mathcal P_n$. Note that if $\mathcal P_n$ is any parallelepiped for $\Lambda_n$, then 
\[\Vol(\mathcal P_n)=[\Lambda_1:\Lambda_n]\Vol(\mathcal P_1).\]
Moreover, since $E_{\langle n\rangle} \cap E_{\tor}(\Q)$ is trivial for $n>1$, we have that 
\[
[\Lambda_1:\Lambda_n]\cdot |E_{\tor}(\Q)|=|E(\Q) / E_{\langle n\rangle}|=\prod_{p\mid n} |E(\F_p)|.
\]
Together, these two equations give that 
\[\frac{\Omega_r}{\Vol(\mathcal P_n)}=\frac{\Omega_r\cdot E_{\tor}(\Q)}{\Vol(\mathcal P_1)\cdot \prod_{p\mid n} |E(\F_p)|}.\]
We may choose $\mathcal P_n$ so that  
\[\sqrt{d_n}\leq [\Lambda_1:\Lambda_n] \sqrt{d_1}.\]
Together with (\ref{EqIncExc1}), these imply that 

\begin{multline}\label{EqIncExc2}
\# \{P=\left(\tfrac{A}{C^2},\tfrac{B}{C^3}\right)\in E(\Q) \mid \widehat h(P) \leq T/4\, \text{and} \, \gcd(C,w)=1 \}\\
\geq \frac{|E_{\tor}(\Q)|}{2^r\sqrt{R_{\Q}(E)}} \cdot \Omega_r
\cdot\left( T^{\frac r2}\cdot \prod_{\substack{p \text{ prime}\\p\mid w}}\left(1-\frac{1}{|E(\F_p)|}\right)
-  {3}^r2^{\omega(w)}T^{\frac{r-1}{2}}\sqrt{d_1}
\right),
\end{multline}
where $\omega(w)$ is the number of distinct prime factors of $w$. Since $2^{\omega(w)}=\mathcal S(w),$ the lemma follows.
\end{proof}

These same arguments can be used to give lower bounds for the number of points of bounded height generated
from any linearly independent points in $E(\Q)$.

\begin{lemma}\label{PropBoundsApprx} Assume the notation and hypotheses above.
Suppose $G$ is a subgroup of $E_{\tor}(\Q)$, and that $\mathcal B:=\{P_1,\dots,P_m\}$ is a set of linearly independent points in $E(\Q)$ listed in ascending order by height. If $T>4d(\mathcal{B}),$ then

\begin{multline*}
\# \{P=\left(\tfrac{A}{C^2},\tfrac{B}{C^3}\right)\in E(\Q) \mid \widehat h(P) \leq T/4 \, \text{and} \, \gcd(C,w)=1 \}\\
\geq \frac{|G|}{ 2^r\sqrt{\widehat h(P_m)^m}}\cdot \Omega_m \left(T^{\frac{m}{2}}\cdot \prod_{\substack{p \text{ prime}\\p\mid w}}\left(1-\frac{1}{|E(\F_p)|}\right) -{3}^mm^2\mathcal S(w)\sqrt{2\widehat h(P_m)}T^{\frac{m-1}{2}} \right).
\end{multline*}
\end{lemma}
\begin{proof} The proof of Lemma~\ref{PropRawBoundsCop} applies with two modifications.
Note that
${d(\mathcal B)} \leq {2m^2\widehat h (P_m)}$, and that the volume of the parallelepiped for $\mathcal B$ satisfies
$\Vol (\mathcal B)\leq \prod_{i=1}^m \widehat h(P_i)^{1/2}\leq \widehat h(P_r)^{\frac{m}{2}}.$
\end{proof}

\begin{proof}[Proof of Theorem~\ref{General}]
By hypothesis, we have that
\[(|u|+w^2)^2\exp(4\delta(E)+d(E))< D< \frac{(|u|+w^2)^2\max(|u|,w^2)^2}{v^4}.\]
Lemma~\ref{PropRawBoundsCop} implies that
\begin{multline}\label{MainCount}
\# \{P=\left(\tfrac{A}{C^2},\tfrac{B}{C^3}\right)\in E(\Q) \mid \widehat h(P) \leq \tfrac14 T_E(-D,Q_{-D}) \, \text{ and } \, \gcd(w,C)=1 \}\\
\geq c(E,Q_{-D})\cdot  T_E(t)^{\frac{r}{2}}- \widehat{c}(E,Q_{-D})T_E(t)^{\frac{r-1}{2}} .
\end{multline}
We show that points $P_1\neq \pm P_2$ in this set map to inequivalent forms when paired with $Q_{-D}=\left(\tfrac{u}{w^2},\tfrac{v}{w^3}\right)\in E_{-D}(\Q)$.
 
Suppose that  $P_1=(\tfrac{A_1}{C_1^2},\tfrac{B_1}{C_1^3}),  P_2=(\tfrac{A_2}{C_2^2},\tfrac{B_2}{C_2^3})\in E(\Q)$ satisfy $\widehat h(P_i) \leq \tfrac14 T_E(-D,Q)$, and let ${F_1:=F_{P_1,Q_{-D}}(X,Y)}$ and $F_2:=F_{P_2,Q_{-D}}(X,Y)$.
Thanks to Theorem~\ref{Silverman_bounds}, we have  that 
\begin{multline}h_W(P_i)\leq 2\left(\widehat h(P_i)+\tfrac18 h_W(j(E))+\tfrac{1}{12}h_W(\Delta(E))+0.973\right)\\
\leq \tfrac12 \log\left|\frac{D}{(|u|+w^2)^2}\right|-\log(2)=\tfrac12 \log\left|\frac{D}{4(|u|+w^2)^2}\right|.
\end{multline}
We observe that  
$\alpha_i=|A_iw^2-uC_i^2|\leq (|u|+w^2)H(P_i)$. By Theorem~\ref{Silverman_bounds}, we have
\begin{equation}\label{HiBound}
H(P_i)=\exp(h_W(P_i))\leq \frac{\sqrt{D}}{2(|u|+w^2)},\end{equation}
 which gives that $\tfrac{\alpha_i}{G_i}\leq\tfrac12\sqrt{D}.$ Hence, we find that
$\frac{\alpha_1}{G_1}\frac{\alpha_2}{G_2}\leq\tfrac{1}{4}D,$
and so by Theorem~\ref{ThmQF}, $F_1(X,Y)$ and $F_2(X,Y)$ are inequivalent, unless 
\begin{equation}\label{EqnAlph}
\frac{\alpha_1}{G_1}=\frac{|A_1w^2-uC_1^2|}{G_1}=\frac{|A_2w^2-uC_2^2|}{G_2}=\frac{\alpha_2}{G_2}.
\end{equation}

Since $\gcd(A_i,C_i)=\gcd(C_i,w)=1$, we have that $G_i=\gcd(\alpha_i,v^2),$ and so $G_i \leq {v}^2.$
Rearranging (\ref{EqnAlph}), we obtain
\begin{equation}\label{eqnRe}{u}(C_1^2{G_2}\pm C_2^2{G_1})={w^2}(A_1{G_2}\pm A_2{G_1}),
\end{equation}
where both signs are the same.
Since ${u}$ and ${w^2}$ are co-prime, we have that 
\begin{equation}\label{eqnDiv}
 {w^2} \, \mid \, (C_1^2{G_2}\pm C_2^2{G_1}) 
 \ \text{ and } \ {u} \, \mid \, (A_1{G_2}\pm A_2{G_1}).
\end{equation}
However, by hypothesis $D\leq \frac{(|u|+w^2)^2\max(|u|,w^2)^2}{v^4},$ and so, combined with (\ref{HiBound}), we find that 
\[|A_i|, C_i^2 \leq H(P_i)< \frac{\max(|u|,w^2)}{2v^2}.\]
This gives that  
\[
|C_1^2{G_2}\pm C_2^2{G_1}| \text{ and } |A_1{G_2}\pm A_2{G_1}| < {\max(|u|,w^2)}.
\]
However,  the divisibility conditions in (\ref{eqnDiv}) imply that at least one of $(C_1^2{G_2}\pm C_2^2{G_1})$ and $(A_1{G_2}\pm A_2{G_1})$ is zero, and therefore, by (\ref{eqnRe}), both are zero. Then we have that $A_1{G_2} =\pm A_2{G_1}$, and $C_1^2{G_2}=\pm C_2^2{G_1}$, where once again both signs are the same. Dividing these terms gives that
$\frac{A_1}{C_1^2}=\frac{A_2}{C_2^2},$ which implies that $P_1=\pm P_2.$ This explains the extra factor of $1/2$ which appears in (\ref{cE}).
This completes the proof.
\end{proof}

\section{An auxilary Diophantine result}\label{DiophantineResult}

Theorems~\ref{MainTheorem}
involves the quadratic twists of the elliptic curves
$$
E^{(a)}: \ \ y^2=x^3-a.
$$ 
Here we prove an auxiliary Diophantine result (see Theorem~\ref{summatory}), motivated by these curves, which will play a central role in the proof of Theorem~\ref{MainTheorem}.
To make this precise, in this section we fix a curve $E^{(a)},$ where $a$ is a positive integer, and
we let $N^{(a)}$ denote its conductor, which is well known to be a multiple of $3$.

Fix an arithmetic progression $h\pmod{4 N^{(a)}},$ where $\gcd(h,4N^{(a)})=1.$  As $X\rightarrow +\infty,$ we aim to count the number of square-free  $0<d<X$ for which there are integer triples
$(m,n,t)$ with
\begin{equation}\label{ranges1}
-dt^2=m^3-an^6,
\end{equation}
where 
\begin{equation}\label{ranges2}
\gcd(t,6am)=1,\ \gcd(n,am)=1,\ m\equiv h\!\!\! \!\pmod{4N^{{(a)}}} {\text {\rm~and~}}n\equiv0\! \!
\! \pmod{4N^{(a)}},
\end{equation}
\begin{equation}\label{ranges3}
T\leq t\leq 2T,\ \ \ M\leq m\leq 2M,\ \ \ N\leq n \leq 2N,\\
\end{equation} 
where\footnote{The function $M_a(X)$ does not depend on the choice of $a$.} $M=M_a(X):=\frac{1}{4}T(X)^{A}\cdot X^{\frac{1}{3}}$ and $N=N_a(X):=\frac{1}{2}{a}^{-\frac{1}{6}}T(X)^{B}\cdot X^{\frac{1}{6}}.$  
Here we assume that
$T:=T(X)$ is a non-decreasing function from $\R^{+}\mapsto [1,\infty),$
and we require that
\begin{equation}\label{ranges4}
0<A<2B < \frac{2}{3}.
\end{equation} 
For  large $X$,  this last  condition  guarantees that the square-free $d$ in (\ref{ranges1}) satisfies $0<d<X.$

For positive square-free integers $d$, we let 
\begin{equation}
N^{(a)}_{h}(d;X,T):=\# \{ (m,n,t)  \ {\text {\rm satisfying}}\ (\ref{ranges1}-\ref{ranges4})\}.
\end{equation}
Theorems~\ref{MainTheorem} will be obtained  from the following summatory asymptotic 
for $N_h^{(a)}(d;X,T).$

\begin{theorem}\label{summatory}
Assume the notation and hypotheses above. As $X\rightarrow +\infty,$ we have
$$
\sum_{1\leq d\leq X} N^{(a)}_{h}(d;X,T)
\asymp_{a} 
X^{\frac{1}{2}}T^{A+B-1} + o_{a,\varepsilon}(X^{\frac{1}{3}+\varepsilon}T^{A+1}).
$$
\end{theorem}

\begin{ThreeRemarks} \ \

\noindent
(1) Theorem~\ref{summatory} illustrates that the
 vast majority of triples $(m,n,t),$ for any given $d,$ have small $t$. If $T(X)=o(X^{\varepsilon}),$ then
 the summation in the theorem is $\gg_{a} X^{\frac{1}{2}-\varepsilon}$.
 Indeed, one can even choose $T(X):=1$ and obtain this asymptotic.
  On the other hand, since
 $-1<A+B-1<0$, the asymptotic is $o(X^{\frac{1}{2}-\varepsilon})$ if $T(X):=X^{C}$ for any positive $C$.
 
 \smallskip
 \noindent
 (2) Assuming the hypotheses of Theorem~\ref{summatory}, an elementary argument (see (\ref{Nbound})) shows that $N_{h}^{(a)}(d;X,T)=O(X^{\varepsilon}).$ Therefore, if $T=T(X)=o(X^{\frac{1}{6(B-2)}})$ (e.g. a log power), then we have
 \begin{equation}\label{refined}
 \# \{ -X<-D<0 \ : \  \exists \left(\frac{u}{w^2},\frac{v}{w^3}\right)\in E^{(a)}_{-D}(\Q)\setminus E^{(a)}_{\tor}(\Q) \ {\text {\rm with}}\
 T\leq v \leq 2T\} \gg_{a,\varepsilon} X^{\frac{1}{2}-\varepsilon}T^{A+B-1}.
 \end{equation}
 
\smallskip
\noindent
(3) 
Soundararajan considered \cite{Soundararajan}  similar Diophantine equations in his work on
torsion in class groups of imaginary quadratic fields, and has results which are analogous to Theorem~\ref{summatory}.\end{ThreeRemarks}

\subsection{The counting function $\rho^{(a)}_m(M)$}

The proof of Theorem~\ref{summatory} requires
the counting function 
\begin{equation}\label{rho}
\rho_{m}^{(a)}(M):= \# \{ n\! \!\!\!\pmod M \ : \ an^6\equiv m^3\pmod{M}\},
\end{equation}
where $a, m$ and $M$ are non-zero integers.
The following lemma gives a closed formula for this function in terms of 
$\mathrm{ord}_{p}(n):=\max\{ t\geq 0 \ : \ p^t\mid n\}$,
Legendre symbols
$\leg{\cdot}{p}$ and the cubic residue symbol
\begin{equation}
\cubic{b}{p}:=\begin{cases} 0 & \ \ \ \ \ {\text {\rm if}}\ p\mid b,\\
-1 & \ \ \ \ \ {\text {\rm if $b$ is not a cubic residue modulo $p$}},\\
1& \ \ \ \ \ {\text {\rm if $b$ is a cubic residue modulo $p$.}}
\end{cases}
\end{equation}

\begin{lemma}\label{rholemma} Assuming the notation above, the following are true.

\noindent
(1) The function $\rho_{m}^{(a)}(M)$ is multiplicative in $M$.

\noindent
(2) If $p=2$, $p\nmid am$ and $\alpha\geq1$, then $\rho_{m}^{(a)}(2^{\alpha})=\widetilde{\rho}_{m}^{(a)}(2^{\alpha}),$ where
\begin{align*}
    \widetilde{\rho}_{m}^{(a)}(2^{\alpha}):=
    \begin{cases}
    1~&\mathrm{if}~\alpha=1,\\
    \\
    1+\displaystyle{\prod_{{q|am\ {\text {\rm prime}}}}}\leg{-1}{q}~&\mathrm{if}~\alpha\geq2.
    \end{cases}
\end{align*}

\noindent
(3) If $p=3$, $p\nmid am$ and $\alpha\geq1$, then $\rho_{m}^{(a)}(3^{\alpha})=\widetilde{\rho}_{m}^{(a)}(3^{\alpha})$, where
\begin{align*}
    \widetilde{\rho}_{m}^{(a)}(3^{\alpha}):=
    \begin{cases}
     \left(1+\leg{a^{-1}m}{3}\right)~&\mathrm{if}~\alpha=1,\\
     \\
    \frac{3}{2}\cdot\left(1+\leg{a^{-1}m}{3}\right)\left(1+\cubic{a^{-1}}{9}\right)~&\mathrm{if}~\alpha\geq2.
    \end{cases}
\end{align*}

\noindent
(4) If $p\geq 5$ is prime, $p\nmid am$ and $\alpha \geq 1$, then $\rho_{m}^{(a)}(p^{\alpha})=\widetilde{\rho}_{m}^{(a)}(p^{\alpha})$, where
\begin{align*}
    \widetilde{\rho}_{m}^{(a)}(p^{\alpha}):=\frac{1}{2}\cdot\left(1+\cubic{a^{-1}}{p}~\right)\left(1+\leg{a^{-1}m}{p}\right)\left(2+\leg{-3}{p}\right).
\end{align*}

\noindent
(5) If $p\geq 5$ is prime, $p\nmid a$, $p|m$ and $\alpha\geq1$, then $\rho_{m}^{(a)}(p^{\alpha})=\widehat{\rho}_{m}^{(a)}(p^{\alpha})$, where
\begin{align*}
    \widehat{\rho}_{m}^{(a)}(p^{\alpha}):=
    \begin{cases}
    2\alpha-1 ~&\mathrm{if}~\alpha\leq3\cdot\mathrm{ord}_{p}(m),\\
    0 ~&\mathrm{if}~\alpha>3\cdot\mathrm{ord}_{p}(m)~\mathrm{and}~2\nmid\mathrm{ord}_{p}(m),\\
    \widetilde{\rho}^{(a)}_{m/p^{\mathrm{ord}_{p}(m)}}(p^{\alpha-3\cdot\mathrm{ord}_{p}(m)}) ~&\mathrm{if}~\alpha>3\cdot\mathrm{ord}_{p}(m)~\mathrm{and}~2|\mathrm{ord}_{p}(m).
    \end{cases}
\end{align*}

\noindent
(6) If $p\geq 5 $ is prime, $p|a$ and $\alpha\geq1$, then 
\begin{align*}
    \rho_{m}^{(a)}(p^{\alpha})=
    \begin{cases}
    p^{\alpha} ~&\mathrm{if}~\alpha\leq3\cdot\mathrm{ord}_{p}(m)~\mathrm{and}~\alpha\leq\mathrm{ord}_{p}(a),\\
    2\cdot(\alpha-\mathrm{ord}_{p}(a))-1 ~&\mathrm{if}~\alpha\leq3\cdot\mathrm{ord}_{p}(m)~\mathrm{and}~\alpha>\mathrm{ord}_{p}(a),\\
    \widehat{\rho}^{(a)}_{m/p^{\mathrm{ord}_{p}(a)}}(p^{\alpha-\mathrm{ord}_{p}(a)}) ~&\mathrm{if}~\alpha>3\cdot\mathrm{ord}_{p}(m)>\mathrm{ord}_{p}(a),\\
     \widetilde{\rho}^{(a)}_{m/p^{\mathrm{ord}_{p}(a)}}(p^{\alpha-\mathrm{ord}_{p}(a)}) ~&\mathrm{if}~\alpha>3\cdot\mathrm{ord}_{p}(m)=\mathrm{ord}_{p}(a).
    \end{cases}
\end{align*}
\end{lemma}
\begin{proof}
Claim (1) follows immediately by the Chinese Remainder Theorem. 

For claim (2), since $2\nmid am$, we may rewrite the congruence equation as the equation
\begin{align*}
    n^6\equiv a^{-1}m^3\pmod{2^{\alpha}}.
\end{align*}
By observation, we have $\rho_{m}^{(a)}(2)=1$. Moreover, for any $\alpha\geq 2$, using Hensel's Lemma, we get $\rho_{m}^{(a)}(2^{\alpha})=\rho_{m}^{(a)}(4)$. Since $(\Z/4\Z)^{\times}$ is cyclic group of order $2$, the congruence is $ 1\equiv am\pmod{4}.$ Hence, this condition is determined by $1+\displaystyle{\prod_{q|am \ {\text {\rm prime}}}}\leg{-1}{q}$. 

Similarly, for claim (3), we consider the equation $n^6\equiv a^{-1}m^3\pmod{3^{\alpha}}.$ If $\alpha=1$, then it turns out that $1\equiv a^{-1}m\pmod{3},$ which corresponds to the factor $1+\leg{a^{-1}m}{3}.$ Furthermore, for any $\alpha\geq2$, using Hensel's Lemma, we only need to consider the equation $n^6\equiv a^{-1}m^3\pmod{9}.$ Since $(\Z/9\Z)^{\times}$ is a cyclic group of order $6$, the equations becomes $1\equiv a^{-1}m^3\pmod{9}$. Hence, we have $\rho_{m}^{(a)}(3^{\alpha})=6$ if $a^{-1}m$ is a quadratic residue modulo $3$ and $a^{-1}$ is a cubic residue modulo $9$. Otherwise, we have $\rho_{m}^{(a)}(3^{\alpha})=0$.

For claim (4), we again apply Hensel's Lemma to get $\rho_{m}^{(a)}(p^{\alpha})= \rho_{m}^{(a)}(p).$ Therefore, it suffices to establish the formula for $\rho_{m}^{(a)}(p)$. It is clear that $\rho_{m}^{(a)}(p)=0$ when $a^{-1}$ is not a cubic residue modulo $p,$ or $a^{-1}m$ is not quadratic residue modulo $p$. Moreover, since $p$ is an odd prime, if $\rho_{m}^{(a)}(p)\neq0$, then the condition $p\equiv 1 \pmod{3}$ and $p>3$ gives $\rho_{m}^{(a)}(p)=6$. Otherwise, we have $\rho_{m}^{(a)}(p)=2$. These two cases determine the last factor $2+\leg{-3}{p}$.

Next, we deal with claims (5) by factoring out prime factor $p$ of $a$ and $m$ to get the results from previous claims (2)--(4). For claim (5), if $\alpha\leq3\cdot\mathrm{ord}_{p}(m)$, then we solve the equation $n^{6}\equiv 0\pmod{p^{\alpha}}$. Hence, we have $\rho_{m}^{(a)}(p^{\alpha})=2\alpha-1$. Moreover, if $\alpha>3\cdot\mathrm{ord}_{p}(m)$, then $2\nmid\mathrm{ord}_{p}(m)$ gives $\rho_{m}^{(a)}(p^{\alpha})=0$ by comparing $p$-valuations on the both sides of the equation. Now, we consider the last case $\alpha>3\cdot\mathrm{ord}_{p}(m)$ and $2|\mathrm{ord}_{p}(m)$. After removing the prime factor from the original equation, we have 
\begin{align*}
    n^{6}\equiv a^{-1}\left(\frac{m}{p^{\mathrm{ord}_p(m)}}\right)^3\pmod{p^{\alpha-3\cdot\mathrm{ord}_p(m)}}.
\end{align*}
Hence, we can use claims (2)--(4) to solve the above equation.
Finally, we apply the same argument to obtain (6). Hence, the result follows from the previous claims (2)--(5). 
\end{proof}

To prove Theorem~\ref{summatory},  we will need to obtain ``average value''
results for $\rho_{m}^{(a)}(M)$. To this end, we must contend with the fact that
the cubic residue symbol $\cubic{\cdot}{p}$ is not multiplicative. However, using the algebraic number theory of the Eisenstein field, we can circumvent this issue by making use of the genuine cubic character
$\cubiccharacter{\cdot}{\pi}.$ 

To make this precise, we let
$\omega:=({-1+\sqrt{-3}})/{2},$ and we employ the ring of Eisenstein integers $\Z[\omega]=\left\{a+b\omega:~a,b\in\Z\right\}$. Let $\pi$ be a prime in $\Z[\omega]$ such that the norm $N(\pi)\neq 3$. Given any $\beta\in\Z[\omega]$ and $k\in\left\{0,1,2\right\}$, the \textit{cubic residue character} of $\beta\pmod \pi$ is defined by
\begin{equation}
\cubiccharacter{\beta}{\pi}:=
   \begin{cases}
   \omega^{k}~&\textrm{if}~\beta^{\frac{N(\pi)-1}{3}}\equiv\omega^{k}\pmod\pi,\\
   0~&\textrm{if}~\pi|\beta.
   \end{cases}
   \end{equation}

\begin{lemma}\label{CubicDictionary}
Assuming the notation above, the following are true.

\noindent
(1) The cubic function $\cubiccharacter{\cdot}{\pi}$ defines a multiplicative character from $\Z[\omega]$ to $\C.$

\noindent
(2) If $N(\pi)=p$ is a prime $p\equiv1\pmod 3,$ then $p=\pi\cdot\overline{\pi}$, where $\overline{\pi}$ is the conjugate of $\pi$ in $\Z[\omega]$, and we have

\begin{align*}
    \cubic{n}{p}=\frac{2}{3}\cdot\cubiccharacter{n}{\pi}+\frac{2}{3}\cdot\cubiccharacter{n}{\overline{\pi}}-\frac{1}{3}.
\end{align*}

\noindent
(3)  If $\pi=p$ is a prime with $p\equiv2\pmod 3$, then for any $n\in\Z$ with $p\nmid n$ we have 
    \begin{align*}
        \cubic{n}{p}=\cubiccharacter{n}{p}.
    \end{align*}
\end{lemma}

\begin{proof}
Given any $\beta_1$ and $\beta_2$ in $\Z[\omega]$. By the definition of  $\cubiccharacter{\cdot}{\pi}$, we obtain
\begin{align*}
     \cubiccharacter{\beta_1\beta_2}{\pi}=\cubiccharacter{\beta_1}{\pi}\cubiccharacter{\beta_2}{\pi}.
\end{align*}
Hence, the claim $(1)$ holds.

Next, we recall the well-known fact that 
\begin{align}\label{ecube}
    n~\mathrm{is~a~cubic~residue~modulo}~\pi \quad \Longleftrightarrow \quad \cubiccharacter{n}{\pi}=1.
\end{align}
Suppose that $n$ is a cubic residue modulo $p$. Then  $n$ is a cubic residue modulo $\pi$ (resp. $\overline{\pi}$). Hence, by the definition of cubic residue symbol and (\ref{ecube}), we have $$\cubic{n}{p}=1=\frac{2}{3}\cdot\cubiccharacter{n}{\pi}+\frac{2}{3}\cdot\cubiccharacter{n}{\overline{\pi}}-\frac{1}{3}.$$
Similarly, if $n$ is a cubic non-residue modulo $p$, then we have $n$ is a cubic non-residue modulo $\pi$ (resp. $\overline{\pi}$). It follows that
\begin{align*}
    \cubic{n}{p}=-1,\quad\mathrm{and}\quad\overline{\cubiccharacter{n}{\pi}}={\cubiccharacter{n}{\overline{\pi}}}\neq1.
\end{align*}
Hence, by the fact that ${\cubiccharacter{n}{\pi}}+\overline{\cubiccharacter{n}{\pi}}=-1$, we get the desired formula.

Last, we deal with the claim (3). Since $p\equiv2\pmod 3$, we know that $p$ is still a prime in $\Z[\omega]$. Hence, the proof of (3) is directly from the fact that every integer is a cubic residue modulo $p$ for $p\equiv2\pmod 3$.
\end{proof}

\subsection{Some average value theorems for $\rho_{m}^{(a)}(M)$}

To prove Theorem~\ref{summatory}, we require asymptotic formulas controlling the
average behavior of the counting functions $\rho_{m}^{(a)}(M)$. To this end, we need to
establish that there is ample cancellation for certain sums arising from Legendre symbols
and the cubic residue symbols. 

Such cancellation can be deduced
using the P\'olya-Vinogradov \cite{Polya, Vinogradov} inequality for any non-principal Dirichlet character $\chi(\cdot)$ modulo $q$
\begin{equation}\label{PV}
\sum_{M\leq n\leq M+X}\chi(n)\ll \sqrt{q}\log q,
\end{equation}
where the implied constant in $\ll$ is absolute.
The following lemma, which 
involves $\tau(n)$, the number of divisors of $n,$ and Euler's totient function $\varphi(n),$ will play a central role in the proof of the two parts of Theorem~\ref{summatory}.

\begin{lemma}\label{SoundLemma1}
If $t$ is an integer for which $\gcd(t,6)=1,$ and let $d_1$, $d_2$, and $d_3$ denote square-free divisors of $t$ for which $d_2\neq1,$ then we have
\begin{align*}
    \sum_{\substack{M\leq m\leq2M\\m\equiv h\!\!\! \!\pmod{4N^{{(a)}}}\\ \gcd(6am,t)=1}}\cubiccharacter{a^{-1}}{d_1}\leg{a^{-1}m}{d_2}\leg{-3}{d_3}\ll_a \tau(t)\sqrt{d_2}\log d_2.
\end{align*}

\end{lemma}
\begin{proof}
Using the orthogonality property of the two Dirichlet characters modulo $4N^{{(a)}}$ to isolate the congruence class $m\equiv h\!\!\pmod{4N^{{(a)}}}$,  we immediately obtain
\begin{align*}
K&:= \sum_{\substack{M\leq m\leq2M\\m\equiv h\!\!\! \!\pmod{4N^{{(a)}}}\\ \gcd(6am,t)=1}}\cubiccharacter{a^{-1}}{d_1}\leg{a^{-1}m}{d_2}\leg{-3}{d_3}\\
  &=\frac{1}{\varphi(4N^{{(a)}})} \cubiccharacter{a^{-1}}{d_1}\leg{-3}{d_3}\sum_{\chi~(\mathrm{mod}~4N^{{(a)}})}\sum_{\substack{M\leq m\leq2M\\ \gcd(6am,t)=1}}\chi(h^{-1}m)\leg{a^{-1}m}{d_2}.
\end{align*}
We now use the elementary fact that $\sum_{d\mid n}\mu(d)=1$ (resp. $0$) if
$n=1$ (resp. $n>1$). Namely, by considering factorizatons of $m$, say $m=fg$, we obtain
\begin{align*}
  K &=\frac{1}{\varphi(4N^{{(a)}})}\cubiccharacter{a^{-1}}{d_1}\leg{-3}{d_3}\sum_{\chi~(\mathrm{mod}~4N^{{(a)}})}\sum_{M\leq m\leq2M}\sum_{f\mid \gcd(6am,t)}\mu(f)\chi(h^{-1}m)\leg{a^{-1}m}{d_2}\\
   &\leq\frac{1}{\varphi(4N^{{(a)}})}\sum_{\chi~(\mathrm{mod}~4N^{{(a)}})}\sum_{f\mid t}\left|\sum_{M/f\leq m\leq2M/f}\chi(g)\leg{g}{d_2}\right|.
\end{align*}
Note that $\chi(\cdot)\leg{\cdot}{d_2}$ is a non-principal character with conductor that is a divisor of $4N^{{(a)}}d_2$. Therefore, the P\'olya-Vinogradov inequality (\ref{PV}) gives
\begin{align*}
    K\ll_{a} \sum_{f\mid t}\sqrt{d_2}\log d_2=\tau(t)\sqrt{d_2}\log d_2.
\end{align*}
This completes the proof. 
\end{proof}

The lemma above establishes cancellation when summing over  $m$, which
appears in an upper parameter of a quadratic residue symbol.
The proof of Theorem~\ref{summatory} also requires the following lemma which guarantees ample cancellation of a similar sum
involving the lower parameters of these symbols. To state this lemma, we let
$\omega(n)$ denote the number of distinct prime factors of $n$, and we let $\widehat\omega(\alpha),$ where $\alpha\in \Z[\omega],$ denote the number of distinct prime factors $\pi\equiv 1~(\mathrm{mod~3}).$ Moreover, we extend the M\"obius function to $\Z[\omega]$ in the natural way.

\begin{lemma}\label{UglyLemma} Suppose that $a$ and $m$ are non-zero integers for which $am$ is not an integral square, and $a$ is not an integral cube.  If $R_1,R_2,R_3\geq 2$ are integers and $T>0$, then we have
\begin{displaymath}
\begin{split}
\Psi_{m}^{(a)}(R_1,R_2, R_3;T)&:= \sum_{\substack{r_1 r_1'\leq 2T/R_1 \\ \gcd(r_1 r_1',6am)=1}}  \
 \sum_{\substack{r_2 r_2'\leq 2T/R_2\\ \gcd(r_2 r_2',6am)=1}}  \
 \sum_{\substack{r_3 r_3'\leq 2T/R_3\\ \gcd(r_3 r_3',6am)=1}}  
   S_{1,r_1}\cdot S_{2,r_2}\cdot S_{3,r_3}\\
   &\ \ \  \ll_a 
     m^{\frac{3}{2}}\prod_{i=1}^3\left(\frac{T}{\sqrt{R_i}}\sqrt{\log m}\right),
\end{split}
\end{displaymath}
where the $r_i$ and $r_i'$ are positive integers, and
\begin{align*}
    S_{1,r_1}:=\mu^2(r_1)\left({2}/{3}\right)^{\widehat\omega(r_1)}\cubiccharacter{a^{-1}}{r_1},\quad
    S_{2,r_2}:=\mu^2(r_2)\leg{a^{-1}m}{r_2},\quad
    S_{3,r_3}:=\mu^2(r_3)\left({1}/{2}\right)^{\omega(r_3)}\leg{-3}{r_3}.
\end{align*}
\end{lemma}
\begin{proof}
Using the coprimality condition $\gcd(6am,r_1r_2 r_3)=1$, we may reformulate the
given triple sum in terms of  characters with explicit conductors that we shall use when applying the P\'olya-Vinogradov inequality (\ref{PV}). Namely, we have
\begin{equation}\label{star}
 \Psi_{m}^{(a)}(R_1,R_2, R_3;T)=G_1\cdot G_2\cdot G_3,
 \end{equation}
where, for each $i,$ we have
\vspace{-.12in}
\[G_i:= \sum_{\substack{r_i r_i'\leq 2T/R_i \\ \gcd(r_i',6am)=1}} \widetilde{S}_{i,r_i},
\vspace{-.12in}
\]
with
\vspace{-.12in}
\begin{align*}
  \widetilde{S}_{1,r_1}:=\mu^2(r_1)&\left({2}/{3}\right)^{\widehat\omega(r_1)}\cubiccharacter{216a^{2}m^3}{r_1}, \ \ \
    \widetilde{S}_{2,r_2}:=\mu^2(r_2)\leg{36am}{r_2}, \ \  \\
    &\widetilde{S}_{3,r_3}:=\mu^2(r_3)\left({1}/{2}\right)^{\omega(r_3)}\leg{-12a^2m^2}{r_3}.
\end{align*}
\vspace{-.005in}
The non-zero summands correspond to cases where $r_1\in \Z[\omega]$ is square-free, and $r_2$ and $r_3$ are square-free in $\Z$.
Therefore, the fact that $\sum_{d\mid n}\mu(d)=1$ for $n=1,$ and $0$ otherwise,
allows us to rewrite this sum as\footnote{The condition $r_i=s_i l_i^2$ includes all factorizations (modulo choices of roots of unity) of $r_i$ over $\Z[\omega]$ when $i=1$, and $r_2$ and $r_3$ over $\Z^{+}.$}
\begin{align*}    
 G_1&:=\sum_{\substack{r_1^{\prime}\leq 2T/R_1\\\gcd(r_1^{\prime},6am)=1}}\sum_{\substack{ r_1\leq2T/r_1^{\prime}}}~\sum_{\substack{l_1^2\mid r_1\\r_1=s_1l_1^2}}\mu(l_1)\left({2}/{3}\right)^{\widehat\omega(r_1)}\cubiccharacter{216a^{2}m^3}{r_1},\\
G_2&:=\sum_{\substack{r_2^{\prime}\leq 2T/R_2\\\gcd(r_2^{\prime},6am)=1}}\sum_{\substack{ r_2\leq2T/r_2^{\prime}}}~\sum_{\substack{l_2^2\mid r_2\\r_2=s_2l_2^2}}\mu(l_2)\leg{36am}{r_2},\\
G_3&:=\sum_{\substack{r_3^{\prime}\leq 2T/R_3\\\gcd(r_3^{\prime},6am)=1}}\sum_{\substack{ r_3\leq2T/r_3^{\prime}}}~\sum_{\substack{l_3^2\mid r_3\\r_3=s_3l_3^2}}\mu(l_3)\left({1}/{2}\right)^{\omega(r_3)}\leg{-12a^{2}m^2}{r_3}.
 \end{align*}
Note that we can view the cubic residue  character $\cubiccharacter{216a^{2}m^3}{\cdot}$ as a non-principal Dirichlet character of order $3$ with conductor dividing $N(216a^{2}m^3)=216^2\cdot|a^4m^6|$. Furthermore, $\leg{36am}{\cdot}$, and $\leg{-12a^2m^2}{\cdot}$ are  non-principal characters with conductors dividing $36\cdot|am|$ and $12\cdot a^2m^2$, respectively. Using the P\'olya-Vinogradov inequality (\ref{PV}) and partial summation, we have that the inner sums of $G_1, G_2,$ and $G_3$ are bounded by $\ll_{a}m^{3}(\log m),\sqrt{m}(\log m),$ and $m(\log m),$ respectively.
Depending on the comparative sizes of $R_i$ and $m$ (i.e. $m$ large), we can also bound
 the inner sum trivially by $\ll \frac{R_i}{l_i^2}$. Hence, we have 
\begin{align*}
    G_1\ll_a\sum_{\substack{r_1^{\prime}\leq 2T/R_1\\\gcd(r_1^{\prime},6am)=1}}\sum_{\substack{ l_1\leq\sqrt{2T/r_1^{\prime}}}}~\min{\left(\frac{R_1}{l_1^2}, m^3(\log m)\right)}\ll_am^{\frac{3}{4}}\frac{m}{\sqrt{R_1}}\sqrt{\log T}.
\end{align*}
Using the same calculation, we have
\begin{align*}
     G_2\ll_am^{\frac{1}{4}}\frac{T}{\sqrt{R_2}}\sqrt{\log m} \ \ \ {\text {\rm and}}\ \ \
     G_3\ll_am^{\frac{1}{2}}\frac{T}{\sqrt{R_3}}\sqrt{\log m}.
\end{align*}
The proof now follows from (\ref{star}).
\end{proof}

Using this lemma, we obtain the following average value result for $\rho_{m}^{(a)}(\cdot)$.

\begin{lemma}\label{SoundLemma2} 
Assume the hypotheses in Theorem~
\ref{summatory}. If 
$T=o(X^{\frac{1}{16}})$ and $\gcd(h,4N^{(a)})=1,$
then
as $X\rightarrow +\infty$, we have
\begin{align*}
    \sum_{\substack{M\leq m\leq2M\\m\equiv h~(\mathrm{mod}~4N^{{(a)}})}}\sum_{\substack{T\leq t\leq 2T\\\gcd(t,6am)=1}}\rho_{m}^{(a)}(t^2)\asymp_{a}  X^{\frac{1}{3}}T^{A+1} +O_a(X^{\frac{7}{24}}T^{\frac{7A}{8}+1}(\log X)^{\frac{39}{8}}).
\end{align*}

\begin{proof}
We recall that $\rho_{m}^{(a)}(M)$ is multiplicative in $M$. Moreover, Lemma~\ref{rholemma} (4) offers
a particularly simple expression for $\rho_{m}^{(a)}(p^{\alpha})$ for primes $p\geq 5$.
To prove (1), we begin by
 restricting to $\gcd(6,t)=1.$ Lemma~\ref{CubicDictionary} gives
\begin{align*}
\Upsilon^{(a)}_{h}(M,T):=    \sum_{\substack{M\leq m\leq2M\\m\equiv h~(\mathrm{mod}~4N^{{(a)}})}}\sum_{\substack{T\leq t\leq 2T\\\gcd(t,6am)=1}}\rho_{m}^{(a)}(t^2)= \sum_{\substack{M\leq m\leq2M\\m\equiv h~(\mathrm{mod}~4N^{{(a)}})}}\sum_{\substack{T\leq t\leq 2T\\\gcd(t,6am)=1}}\sum_{\substack{\theta\mid t\\r_i\mid t,i=2,3}}S_{1,\theta}\cdot S_{2,r_2}\cdot S_{3,r_3},
\end{align*}
where we recall that
\begin{align*}
    S_{1,\theta}=\mu^2(\theta)\left(\frac{2}{3}\right)^{\widehat\omega(\theta)}\cubiccharacter{a^{-1}}{\theta},\quad
    S_{2,r_2}=\mu^2(r_2)\leg{a^{-1}m}{r_2},\quad
    S_{3,r_3}=\mu^2(r_3)\left(\frac{1}{2}\right)^{\omega(r_3)}\leg{-3}{r_3}.
\end{align*}
For convenience, we let
$$
\Upsilon^{(a)}_{h}(M,T)= Y^{(a)}_{h,0}(M,T) + Y^{(a)}_{h,1}(M,T),
$$
where $Y_{h,0}^{(a)}(M,T)$ consists of the summands where
 $\theta=r_2=r_3=1,$ and $Y_{h,1}^{(a)}(M,T)$ denotes the remaining terms.
 We find that $Y_{h,0}^{(a)}(M,T)$ satisfies
\begin{align*}
  Y_{h,0}^{(a)}(M,T)&:=  \sum_{\substack{M\leq m\leq2M\\m\equiv h~(\mathrm{mod}~4N^{{(a)}})}}\sum_{\substack{T\leq t\leq 2T\\\gcd(t,6am)=1}}1=\sum_{\substack{M\leq m\leq2M\\m\equiv h~(\mathrm{mod}~4N^{{(a)}})}}\left(T\frac{\varphi(6am)}{|6am|}+O(\tau(6am))\right)\\ \ \ \\
  &\asymp_{a}MT+O_a((M+T)\log M).
\end{align*}
Next, we estimate $Y_{h,1}^{(a)}(M,T)$, which consists of those summands with $\theta \cdot r_2\cdot r_3\neq1~(\mathrm{mod}~\Z[\omega]^{\times}).$ 
Let $r_1:=\theta$. Then we separate the estimate of $Y_{h,1}^{(a)}(M,T)$ into two pieces by truncating the divisor of $t$: (1) $1\leq r_i\leq R_i$ and (2) $r>R_i$. Hence, we can rewrite $Y_{h,1}^{(a)}(M,T)$ 
\begin{align*}
    U_1+U_2:=\sum_{\substack{M\leq m\leq2M\\m\equiv h~(\mathrm{mod}~4N^{{(a)}})}}\ \sum_{\substack{T\leq t\leq 2T\\\gcd(t,6am)=1}}\left(\sum_{\substack{r_i\mid t,i=1,2,3\\1\leq r_i\leq R_i}}S_{1,r_1}\cdot S_{2,r_2}\cdot S_{3,r_3}+\sum_{\substack{r_i\mid t,i=1,2,3\\r_i\leq t/R_i}}S_{1,t/r_1}\cdot S_{2,t/r_2}\cdot S_{3,t/r_3}\right).
\end{align*}
Moreover, we consider $t=r_ir_i^{\prime}$ and rewrite $U_2$
\begin{align*}
    \sum_{\substack{M\leq m\leq2M\\m\equiv h~(\mathrm{mod}~4N^{{(a)}})}}\ \sum_{\substack{r_i\leq 2T/R_i\\\gcd(r_i,6am)=1}}\ \sum_{\substack{\max{(T/r_i,R_i)}\leq r_i^{\prime}\leq2T/r_i\\\gcd(r_i^{\prime},6am)=1}}S_{1,r_1^{\prime}}\cdot S_{2,r_2^{\prime}}\cdot S_{3,r_3^{\prime}}.
\end{align*}
Lemma~\ref{SoundLemma1} (1) asserts that
\begin{align}\label{U1}
U_1\ll \sqrt{R_2}\log R_2\sum_{T\leq t\leq 2T}\tau(t)^4\ll T\sqrt{R_2}(\log X)^6.
\end{align}

We now estimate $U_2$ by splitting the sum into two parts.
\begin{align*}
U_2&=\sum_{\substack{M\leq m\leq2M\\m\equiv h~(\mathrm{mod}~4N^{{(a)}})}}\sum_{\substack{r_i\leq 2T/R_i\\\gcd(r_i,6am)=1}}\sum_{\substack{\max{(T/r_i,R_i)}\leq r_i^{\prime}\leq2T/r_i\\\gcd(r_i^{\prime},6am)=1}}S_{1,r_1^{\prime}}\cdot S_{2,r_2^{\prime}}\cdot S_{3,r_3^{\prime}}\\
\\ &\ll_a \sum_{\substack{M\leq m\leq2M\\a:\mathrm{non cube,~}am\neq\square}}\sum_{\substack{r_i\leq 2T/R_i\\\gcd(r_i,6am)=1}} \sum_{\substack{\max{(T/r_i,R_i)}\leq r_i^{\prime}\leq2T/r_i\\\gcd(r_i^{\prime},6am)=1}}S_{1,r_1^{\prime}}\cdot S_{2,r_2^{\prime}}\cdot S_{3,r_3^{\prime}} 
+ \sum_{\substack{M\leq m\leq2M\\a:\mathrm{cube,~or~}am=\square}}\sum_{\substack{r_i\leq 2T/R_i\\\gcd(r_i,6am)=1}} \frac{T}{r_i}.
\end{align*}
Moreover, since $m\leq2M$ and $\log M\ll\log X$,  Lemma~\ref{UglyLemma} gives
\begin{align}\label{U2}
U_2\ll_a M^{
\frac{5}{2}}\prod_{i=1}^3\left(\frac{T}{\sqrt{R_i}}\sqrt{\log X}\right)+\sqrt{M}T^3(\log X)^3.
\end{align}
To balance the exponents of $M$ and $\log X$ in (\ref{U1}) and (\ref{U2}),
we take $R_1=R_2=R_3=M^{\frac{5}{4}}/(\log X)^{\frac{9}{4}},$ which in turn gives
$$
Y_{h,1}^{(a)}(M,T) \ll_a M^{\frac{5}{8}}T^3(\log X)^{\frac{39}{8}}.
$$
By the hypothesis on $T$, we have $T^2\ll M^{\frac{3}{8}},$ and so
$$
Y_{h,1}^{(a)}(M,T) \ll_a M^{\frac{7}{8}}T(\log X)^{\frac{39}{8}}.
$$
In view of the asymptotic for $Y_{h,0}^{(a)}(M,T),$ we see that $Y_{h,1}^{(a)}(M,T)$ is the error term for $\Upsilon^{(a)}(M,T),$ completing the proof. 
\end{proof}
\end{lemma}

\subsection{Proof of Theorem~\ref{summatory}}\label{ProofThm2}

The claimed summatory formula for $N_h^{(a)}(d;X,T)$ counts the number of 3-tuples $(m,n,t)$ satisfying  (\ref{ranges1}-\ref{ranges4}), which guarantees that $(m^3-an^6)/t^2$ is square-free and negative. 
We shall show that this count is well approximated by those tuples, where these values are not divisible
by squares of small primes $p$. Namely, we let
\begin{equation}
    C^{(a)}_h(X,T):=\#\left\{(m,n,t) \ {\text {\rm satisfies (\ref{ranges1}-\ref{ranges4}) and}}\ p^2\nmid \frac{(m^3-an^6)}{t^2}~\mathrm{for~}p\leq\log X\right\}.
\end{equation}
By hypothesis, for sufficiently large $X,$ the dependence on the ranges for $m$ and $n$ on $X$ guarantee that $m^3-an^6$ is negative. Therefore, for large $X,$
we find that $C^{(a)}_h(X,T)$ is a good approximation, provided that $E_{h,2}^{(a)}(X,T)$ and $E_{h,3}^{(a)}(X,T)$ are small, where $Z(X,T):=X^{\frac{1-C}{2}}(\log X)^{\frac{2}{3}}$ and
\begin{align*}
    &E_{h,2}^{(a)}(X,T):=\\
    &\ \ \  \ \ \ \ \#\left\{(m,n,t) \ {\text {\rm satisfies (\ref{ranges1}-\ref{ranges4})}}~\mathrm{~and~}p^2\mid 
    \frac{(m^3-an^6)}{t^2}~\mathrm{for~some~}\log X<p\leq Z(X,T) \right\},\\
    &E_{h,3}^{(a)}(X,T):=\#\left\{(m,n,t)\ {\text {\rm satisfies (\ref{ranges1}-\ref{ranges4})}}~\mathrm{~and~}p^2\mid \frac{(m^3-an^6)}{t^2}~\mathrm{for~some~}p>Z(X,T)\right\}.
\end{align*}

We first obtain an  asymptotic formula for $C^{(a)}_h(X,T)$.
We define the product of small primes $P(X):=\prod_{p\leq\log X}p$. Then we have
\begin{align*}
    C^{(a)}_h(X,T)=\sum_{m,t}\sum_{\substack{N\leq n\leq2N\\\gcd(n,am)=1\\ n\equiv0\! \!
\! \pmod{4N^{(a)}}\\an^6\equiv m^3~\mathrm{(mod}~t^2)}}\sum_{l^2\mid\gcd((an^6-m^3)/t^2,P(X)^2)}\mu(l)=\sum_{m,t}\sum_{\substack{l|P(X)\\\gcd(l,am)=1}}\mu(l)\sum_{\substack{N\leq n\leq2N\\n\equiv0\! \!
\! \pmod{4N^{(a)}}\\an^6\equiv m^3~\mathrm{(mod}~l^2t^2)}} 1.
\end{align*}
To be clear, the outer sum in both expressions above is over pairs $(m,t)$ satisfying (\ref{ranges1}-\ref{ranges4}).
Since $T\leq X^{\frac{1}{22}}$, for any $\varepsilon>0$ we can bound the inner sum by
\begin{align*}
    \frac{N}{4N^{(a)}l^2t^2}\rho_{m}^{(a)}(l^2t^2)+O_a(\rho_{m}^{(a)}(l^2t^2))
    = \frac{N}{4N^{(a)}l^2t^2}\rho_{m}^{(a)}(l^2t^2)+O_a(X^{\varepsilon}).
\end{align*}
Here we used the fact that $\rho_m^{(a)}(t^2)=O_a(X^{\varepsilon}),$ which follows by multiplicativity and the fact that
$\omega(t)=O(t^{\varepsilon}).$
Hence, we have
\begin{align*}
    C^{(a)}_h(X,T)&=\sum_{(m,t)}\sum_{\substack{l|P(X)\\\gcd(l,am)=1}}\mu(l)\left(\frac{N}{4N^{(a)}l^2t^2}\rho_{m}^{(a)}(l^2t^2)+O_a(X^{\epsilon})\right)\\
    &=\frac{N}{4N^{(a)}t^2}\sum_{m,t}\rho_{m}^{(a)}(t^2)\sum_{\substack{l|P(X)\\\gcd(l,am)=1}}\frac{\mu(l)}{l^2}\rho_{m}^{(a)}\left(\frac{l}{\gcd(t,l)}\right)+O_a(\tau(P(X))X^{\varepsilon})\\
   &\asymp_a \frac{N}{T^2}\sum_{m,t} \rho_{m}^{(a)}(t^2)+O_a(X^{\varepsilon}).
\end{align*}
Applying Lemma \ref{SoundLemma2}, and by summing over $m$ and $t$, we obtain
\begin{align*}
    C^{(a)}_h(X,T)\asymp_a\frac{MN}{T}+O_a(MTX^{\varepsilon})
    \asymp_{a} 
X^{\frac{1}{2}}T^{A+B-1} + o_{a,\varepsilon}(X^{\frac{1}{3}+\varepsilon}T^{A+1}).
\end{align*}

Since the asymptotic for $C^{(a)}_h(X,T)$ above is the conclusion of the theorem, it suffices 
to show that $E_{h,2}^{(a)}(X,T)$ and $E_{h,3}^{(a)}(X,T)$ are of lower order. We now bound $E_{h,2}^{(a)}(X,T)$ by the following estimate
\begin{align*}
    E_{h,2}^{(a)}(X,T)=\sum_{m,t}\sum_{\log X\leq p\leq Z}\sum_{\substack{N\leq n\leq2N\\n\equiv0\! \!
\! \pmod{4N^{(a)}}\\an^6\equiv m^3~\mathrm{(mod}~t^2p^2)}}1&\ll_a \sum_{m,t}\sum_{\log X\leq p\leq Z}\left(\frac{N}{t^2p^2}\rho_{m}^{(a)}(t^2p^2)+O_a(\rho_{m}^{(a)}(t^2))\right)\\
    &\ll_a \sum_{m,t}\left(\frac{N\rho_{m}^{(a)}(t^2)}{T^2\log X}+o_a\left(\left(\frac{X}{T}\right)^{\frac{1}{3}}\rho_{m}^{(a)}(t^2)\right)\right).
\end{align*}
We used the facts that $\rho_{m}^{(a)}(t^2p^2)\leq 6\rho_{m}^{(a)}(t^2)$ (see Lemma~\ref{rholemma}), $1/t^2\leq1/T^2$, $1/p^2\leq 1/\log X$ and $(T/X)^{\frac{1}{3}}=o(1).$
Applying Lemma \ref{SoundLemma2}, and by summing over $m$ and $t$, we have the lower order asymptotic
\begin{align*}
    E_{h,2}^{(a)}(X,T)\ll_a\frac{MN}{T\log X}+o_{a}\left(\frac{X}{T}^{\frac{1}{3}}\sum_{m,t}\rho_{m}^{(a)}(t^2)\right)\ll_a
    \frac{ X^{\frac{1}{2}}T^{A+B-1}}{\log X}
     + o_a(X^{\frac{1}{3}}T^{-1}).
\end{align*}

Finally, we estimate $E_{h,3}^{(a)}(X,T)$ by using the arithmetic of number fields. Let $p>Z$ be prime, and suppose $d=p^2b$. Then we have $m^3=an^6-p^2b t^2$ and $b\ll\frac{X}{Z^2}=T/{(\log X)^{\frac{4}{3}}}$. Fix $m$ in $[M,2M]$ and with the condition of $b$, we claim that the number of choices for $n$ and $t$ is bounded by $o_a(m)$. Hence, by the claim, we have the lower order asymptotic
\begin{align*}
    E_{h,3}^{(a)}(X,T)\ll_a \frac{X}{Z^2}\sum_{M\leq m\leq 2M}\tau(m)\ll \frac{X}{Z^2}M\log X\ll_{a,\varepsilon} o_{a,\varepsilon}(X^{\frac{1}{3}+\varepsilon}T^{A+1}).
\end{align*}
Now, we prove the claim by factoring the equation $m^3=an^6-p^2b t^2$ in $\Q(\sqrt{a},\sqrt{b})$; namely,
\begin{equation}\label{Nbound}
    (m)^3=(\sqrt{a}n^3+pt\sqrt{b})(\sqrt{a}n^3-pt\sqrt{b}).
\end{equation}
Since $\gcd(m,n)=\gcd(m,a)=1$ and $m$ is odd, we have two coprime factors. Therefore, the number of choices for $n$ and $t$
is bounded by the total number of factorizations of the ideal $(m)$, which is $o_a(m)=O(X^{\varepsilon})$.

\section{Proof of Theorems~\ref{MainTheorem}}\label{MainTheoremProof}

We begin with an elementary lemma concerning the suitability of  points on quadratic twists.

\begin{lemma}\label{criterion}
Assume the hypotheses of Theorem~\ref{summatory}, and suppose that $T=o(X)$. Then the following are true.
\begin{enumerate}
\item
If $-D$ is odd, $Q_{-D}=(\frac{m}{n^2},\frac{2t}{n^3})\in E_{-D}^{(a)}(\Q),$ where
$(m,n,t)$ satisfies (\ref{ranges1}-\ref{ranges4}) with $d=D,$ then $Q_{-D}$ is suitable in the sense of (\ref{suitable}) when $X$ is sufficiently large.

\item
If $-D=-4D_0$, where $D_0\equiv 1, 2\pmod 4$ is square-free, $Q_{-D}=(\frac{m}{n^2},\frac{t}{n^3})\in E_{-D}^{(a)}(\Q),$ where $(m,n,t)$ satisfies (\ref{ranges1}-\ref{ranges4}) with $d=D_0,$  then $Q_{-D}$ is suitable in the sense of (\ref{suitable}) when $X$ is sufficiently large.
\end{enumerate}
\end{lemma}
\begin{remark}
Since we assume that $v$ is even when $-D$ is odd,
we choose to use $m, n,$ and $t$ instead of $u, v,$ and $w$ to avoid confusion and enjoy the convenience of working with a single equation.
\end{remark}
\begin{proof}
For brevity, we only consider  when $-D$ is odd, as the same method applies to the other case. 
Recalling the convention in (\ref{twist}), by clearing denominators and dividing by 4 we obtain
\begin{equation}\label{C1}
 -Dt^2=m^3-an^6.
\end{equation}
Furthermore, since $Q_{-D}$ is suitable, we have
\begin{align}\label{suitabledef}
    (|m|+n^2)\exp\left(\delta(E^{(a)})+d(E^{(a)})\right)<D<\frac{(|m|+n^2)^2\max(|m|,n^2)^2}{16t^4}.
\end{align}
We first consider the right hand inequality above using (\ref{C1}).   We find that
\begin{align*}
    t^4<\frac{(|m|+n^2)^2\max(|m|,n^2)^2}{16}\cdot\frac{t^2}{an^6-m^3}\quad \Longleftrightarrow\quad t^2 <\frac{(|m|+n^2)^2\max(|m|,n^2)^2}{16(an^6-m^3)}.
\end{align*}
Recalling that $M:=T^{A}X^{\frac{1}{3}}$ and $N:=T^BX^{\frac{1}{6}}$ in (\ref{ranges3}),
we have $N\gg_a\sqrt{M}$ as $X\rightarrow +\infty.$  Therefore this inequality holds for sufficiently large $X$.
For the left hand inequality in (\ref{suitabledef}), the desired claim follows similarly from (\ref{C1}), as
\begin{align*}
      t^2<\frac{an^6-m^3}{ (|m|+n^2)\exp\left(\delta(E^{(a)})+d(E^{(a)})\right)}\quad \Longleftrightarrow\quad |t|<\sqrt{\frac{an^6-m^3}{(|m|+n^2)\exp\left(\delta(E^{(a)})+d(E^{(a)})\right)}}.
\end{align*}
\end{proof}

We require an explicit lower bound for the ratio
of constants defined by (\ref{cE}) and (\ref{cED}) for those points $Q_{-D}$ considered in this lemma.

\begin{lemma}\label{lambda}
Assume the hypotheses of Theorem~\ref{summatory}, and let $T=O(1).$ If $X$ is sufficiently large, then $Q_{-D}$ is suitable in the sense of
(\ref{suitable}), and we have
$$
 c(E^{(a)},Q_{-D})> 
 \frac{1}{5} \cdot\frac{c(E^{(a)})}{\log \log D}.
 $$
 \end{lemma}
 \begin{remark}
The proof of Theorem~\ref{MainTheorem} holds for any $T=o(X),$
 giving $\gg_{a.\varepsilon} X^{\frac{1}{2}-\varepsilon}$ many discriminants with class number lower bounds. For example, we may let $T=(\log X)^C$, where $C>0.$ However the multiplicative constant in the effective class number lower bound would have to be modified
 by following the proof of Lemma~\ref{lambda}.
 \end{remark}

\begin{proof} 
Lemma~\ref{criterion} guarantees that $Q_{-D}$ is suitable for large $X.$ 
Turning to the claimed inequality, we begin by noting that (\ref{cE}) and (\ref{cED}) give
$$
 \frac{c(E^{(a)},Q_{-D})}{c(E^{(a)})}=\prod_{\substack{p\ \mathrm{ prime}\\p\mid n}}\left(1-\frac{1}{|E^{(a)}(\F_p)|}\right).
$$
As mentioned in the proof of Lemma~\ref{PropRawBoundsCop}, $|E^{(a)}(\F_p)|$ includes the point at infinity, and does not require that $p$
is a prime of good reduction for $E^{(a)}$.
For the small primes $p\in \{2, 3, 5\}$, we have $|E^{(a)}(\F_p)|=p+1.$  Therefore,
the Hasse bound for trace of Frobenius for elliptic curves implies
$$
 \frac{c(E^{(a)},Q_{-D})}{c(E^{(a)})}> 
 \frac{5}{12}\cdot \prod_{\substack{p\mid n\\ \ p\geq 7\ {\text {\rm prime}}}}\left(1-\frac{1}{p+1-2\sqrt{p}}\right)=
 \frac{5}{12}\cdot \prod_{\substack{p\mid n\\ \ p\geq 7\ {\text {\rm prime}}}} \mathcal{F}(p)\cdot \left(1-\frac{1}{p}\right),
$$
where $\mathcal{F}(p):=\frac{p}{p-1}\left(1-\frac{1}{p+1-2\sqrt{p}}\right).$
Since $0<\mathcal{F}(p)<1$ and rapidly monotonically tends to 1 as $p\rightarrow +\infty$, we have 
$$
\prod_{\substack{p\mid n\\ \ p\geq 7\ {\text {\rm prime}}}} \mathcal{F}(p)> \prod_{p\geq 7 \ {\text {\rm prime}}} \mathcal{F}(p) >\frac{9}{20}.
$$
Therefore, we have
$$
\frac{c(E^{(a)},Q_{-D})}{c(E^{(a)})}> 
 \frac{3}{16} \cdot \prod_{\substack{p\mid n\\ \ p\geq 7\ {\text {\rm prime}}}}\left(1-\frac{1}{p}\right).
$$

We are left with the problem of obtaining a lower bound for product over primes $p\geq 7$ (if any) which divide $n$ above.
Since the Euler factors increase monotonically to 1 with the primes, and $N\leq n\leq 2N,$
we may bound this product from below with a product over sufficiently many consecutive primes $p\geq 7.$
Namely, if $p_1=2,\  p_2=3,\dots$ are the primes in order, then for large $X$ we have
\begin{equation}\label{ajax}
\frac{c(E^{(a)},Q_{-D})}{c(E^{(a)})}> 
 \frac{3}{16} \cdot \prod_{i=4}^{\kappa(X)} \left(1-\frac{1}{p_i}\right)=\frac{45}{64}\cdot
 \prod_{i=1}^{\kappa(X)} \left(1-\frac{1}{p_i}\right),
 \end{equation}
where  $\kappa(X):=[\log_2(2TX^{\frac{1}{2}})].$ 

For $x>1$, a classical theorem of Rosser and Schoenfeld \cite{RS1, RS2} unconditionally asserts that
$\pi(x)\leq 1.255056\cdot \frac{x}{\log x},$
where $\pi(x)$ is the usual prime counting function. Therefore, if $X$ is sufficiently large, then the fact that $T=O(1)$ implies  
\begin{equation}\label{effectivePNT}
p_{\kappa(X)}\leq \frac{[\log_2 (X^{\frac{1}{2}})]^2}{1.25}=:\lambda(X).
\end{equation}
Let $x>1$, then we recall the effective version of Merten's Theorem (see e.g. \cite[Eq. (3.27)]{RS1})
\begin{align*}
    \prod_{\ p\leq x\ {\text {\rm prime}}}\left(1-\frac{1}{p}\right)>\frac{e^{-\gamma}}{\log x}\left(1-\frac{1}{\log^2 x}\right),
\end{align*}
where $\gamma\approx  0.5772$ is the Euler-Mascheroni constant.
\noindent
Combining this with (\ref{ajax}) and (\ref{effectivePNT}) gives 
$$
\frac{c(E^{(a)},Q_{-D})}{c(E^{(a)})}> 
 \frac{45}{64} \cdot \frac{e^{-\gamma}}{\log (\lambda(X))}\left(1-\frac{1}{\log(\lambda(X))^2}\right)>
0.21586\cdot\frac{1}{\log \log X}.
$$
To complete the proof we must relate those discriminants $-D$ obtained from (\ref{ranges1}-\ref{ranges4}) to $X.$ Namely, we need to consider the following equation
\begin{equation*}
-dt^2=m^3-an^6,
\end{equation*}
where $d:=D$ or $D_0$ depending on the parity of $D$.  
Since $a$ is a positive integer and $T=O(1)$, the growth conditions of $m$ and $n$ imply
$$
d=\frac{an^6-m^3}{t^2}\geq  \frac{aT^{6B}X-T^{3A}X}{
4T^2}\gg_{T} X,
$$
which in turn, for every $\varepsilon>0$, gives
$1/\log \log D \leq (1+\varepsilon)/\log \log X$ for large $X$.
Therefore, for large $X$ we obtain
$$
 c(E^{(a)},Q_{-D})\geq 0.2158 \cdot\frac{c(E^{(a)})}{\log\log D}.
$$
\end{proof}

\begin{proof}[Proof of Theorem~\ref{MainTheorem}]

Suppose that $a$ is a non-zero integer, and that
$-D<0$ is a fundamental discriminant for which there is a rational point $Q_{-D}=(\frac{m}{n^2},\frac{2t}{n^3})\in E_{-D}^{(a)}(\Q),$ when $D$ odd (resp. $Q_{-D}=(\frac{m}{n^2},\frac{t}{n^3})\in E_{-D}^{(a)}(\Q),$ when $D=4D_0$ even).  We may assume that $t$ is non-zero, and so
$Q_{-D}$ is not a 2-torsion point.
Furthermore, it is well known that at most finitely many twists of $E^{(a)}$ (see Proposition 1 of 
\cite{GouveaMazur}) have a torsion point with order $\neq 2$. Therefore, apart from possibly finitely many $-D$, we have that  $Q_{-D}$ has infinite order.

Recalling that our models are of the form (\ref{twist}), 
we find directly that $(m,n,t)$ satisfies (\ref{ranges1}), giving a solution to
\begin{equation}\label{aha}
-dt^2=m^3-an^6,
\end{equation}
where $d:=D$ or $D_0$ depending on the parity of $D$.  

We let ${W}\left(E^{(a)}\right)\in \{\pm 1\}$ be the sign of the functional equation
for the Hasse-Weil $L$-function $L(E^{(a)},s)$.  Recall that $N^{(a)}$ is the conductor
of $E^{(a)}.$  If we have
$\gcd(-D,N^{(a)})=1,$ then (see p. 3 of \cite{GouveaMazur}) the sign of the functional equation for
the quadratic twist $L(E^{(a)}_{-D},s)$ is 
\begin{equation}
W\left(E_{-D}^{(a)}\right)=\leg{-D}{N^{(a)}}\cdot W\left(E^{(a)}\right).
\end{equation}
Therefore, the Parity Conjecture implies that $r_{\Q}(E_{-D}^{(a)})$
is even when $\leg{-D}{N^{(a)}}=W\left(E^{(a)}\right).$ In particular, any
 triple $(m,n,t)$ also satisfying (\ref{ranges1}-\ref{ranges4}) 
conditionally has $r_Q(E_{-D}^{(a)})\geq 2.$

We now apply Theorem 3.1, with $A<2B<2/3$ and $T=O(1),$ with any $h\pmod{4N^{(a)}}$ for which (\ref{aha})
gives $\leg{-d}{N^{(a)}}=W\left(E^{(a)}\right).$ 
Theorem~\ref{summatory} and 
\begin{equation}\label{squaresum}
    \sum_{1\leq d\leq X} N^{(a)}_{h}(d;X,T)\asymp_{a} X^{\frac{1}{2}}.
  \end{equation}
By repeating the argument for (\ref{Nbound}), we have that $N^{(a)}_{h}(d;X,T)=O_{\varepsilon}(X^{\varepsilon}).$  
Therefore, we obtain  
\begin{align*}
     \# \left \{ -X<-D<0 \ : \ r_{\Q}(E_{-D}^{(a)})\geq 1 \right\} \gg_{a,\varepsilon} X^{\frac{1}{2}-\varepsilon}.
\end{align*}
Again, the Parity Conjecture allows us to further require that $r_{\Q}(E_{-D}^{(a)})\geq 2$.

This lower bound
produces $\gg_{a,\varepsilon}X^{\frac{1}{2}-\varepsilon}$ many discriminants $-X<-D<0$ for which $E_{-D}^{(a)}$ has an explicit  infinite order rational
point $Q_{-D}$.
Lemma~\ref{lambda}
 guarantees that $Q_{-D}$ is suitable in the sense of (\ref{suitable}), and also gives
$$c(E^{(a)},Q_{-D}) \geq  0.2158
\cdot\frac{c(E^{(a)})}{\log \log D} > \frac{1}{5}
\cdot\frac{c(E^{(a)})}{\log \log D}.
$$
Therefore, Theorem~\ref{MainTheorem} follows directly from Theorem~\ref{General}.
\end{proof}

\end{document}